\documentclass[11 pt]{amsart}
\usepackage{amsfonts}
\usepackage{latexsym,amsmath,amsfonts}
\usepackage{amssymb}
\usepackage{graphicx}
\usepackage[dvips]{color}
\newtheorem{theorem}{Theorem}[section]
\newtheorem{thm}[theorem]{Theorem}
\newtheorem{pro}[theorem]{Proposition}

\newtheorem{lemma}[theorem]{Lemma}
\newtheorem{remark}[theorem]{Remark}

\numberwithin{equation}{section}

\def\m{\medskip}

\def\cal{\mathcal }
\def\R{\mathbb R}
\def\N{\mathbb N}
\def\Z{\mathbb Z}

\def\mathscr{\mathcal }
\def\CJ{\mathcal J}
\def\CL{\mathcal L}

\def\m{\mathbf m}
\def\1{\mathbf 1}

\newcommand{\blambda}{{\boldsymbol{\lambda}}}

\newcommand{\balpha}{{\boldsymbol{\alpha}}}
\newcommand{\brho}{{\boldsymbol{\rho}}}

\begin{document}

\title[Higher dimensional Frobenius problem]{Higher dimensional  Frobenius problem:
Maximal saturated cone, growth function and rigidity}

\author{Ai-hua Fan} \address{Department of Mathematics and Stochastic, Hua Zhong Normal University, Wuhan 430072,
China \& UMR 7352,  CNRS, Universit\'e de Picardie, 33 rue
Saint Leu, 80039 Amiens, France} \email{ ai-hua.fan@u-picardie.fr}

\author{Hui Rao} \address{Department of Mathematics and Stochastic, Hua Zhong Normal University, Wuhan 430072,
China} \email{ hrao@@mail.ccnu.edu.cn}

\author{Yuan Zhang$\dagger$} \address{Department of Mathematics and Stochastic, Hua Zhong Normal University, Wuhan 430072,
China} \email{ hrao@@mail.ccnu.edu.cn}
\date{} 
\thanks {The work is partially supported by NSFC No. 11171128, NSFC No. 11431007, NSFC No.11471132.
It is also partially supported by the self-determined research funds of CCNU (No. CCNU14Z01002) from the basic research and operation of MOE. }

\thanks{
 {\indent\bf Key words and phrases:}\ Frobenius problem, saturated cone, entropy, directional growth function.}

\thanks{$\dagger$ The correspondence author.}

\begin{abstract} We  consider $m$ integral
vectors $X_1,\dots,X_m \in \mathbb{Z}^s$  located in a half-space of $\mathbb{R}^s$ ($m\ge s\geq 1$) and study the structure of the additive semi-group
$X_1 \mathbb{N} +\dots + X_m \mathbb{N}$. We introduce and study
maximal saturated cone and directional growth function which describe some aspects of the structure of the semi-group.  When the vectors
$X_1, \cdots, X_m$ are located in a fixed hyperplane, we obtain  an explicit formula for the directional growth function and we show that  this function
completely characterizes the defining data $(X_1, \cdots, X_m)$ of the semi-group. The last result
will be applied to
 the study of Lipschitz equivalence of Cantor sets (see \cite{RZ14}).
\end{abstract}

\maketitle
\section{Introduction}

We consider $m$ integral
vectors $X_1,\dots,X_m$   in the lattice $\mathbb{Z}^s$ ($m\ge 2$, $s\geq 1$) which are assumed to be in a half-space. That is to say,  there is a vector
$\balpha\in {\mathbb R}^s$ such that
$\langle X_j, \balpha\rangle>0$  for all $j=1,\dots, m$
where $\langle \cdot, \cdot \rangle$ denotes the inner product on the Euclidean space $\R^s$.
We also assume that $X_1,\dots,X_m$  span the vector space $\mathbb{R}^s$. But  $X_j$'s may not be distinct.
 Let
\begin{equation}\label{equ9}
\CJ=X_1\mathbb{N}+\dots+X_m\mathbb{N}
\end{equation}
be the  semi-group generated by $X_1,\dots, X_m$ where $\N=\{0,1,2,\dots\}$ denotes the set of natural numbers.   By \emph{higher dimensional Frobenius problem} we mean the study of the
structure of the semi-group $\CJ$ defined by (\ref{equ9}). This will be our main concern in the present paper.

In the one-dimensional case (i.e. $s=1$),  we are given $m$ relatively
prime positive integers $a_1, \cdots , a_m$ instead of $X_1, \cdots, X_m$. Then $$
{\mathcal J}=a_1\mathbb{N}+\dots+a_m\mathbb{N}
$$
which is the set of all natural numbers representable as a non-negative integer combination of $a_1, \cdots, a_m$. Sylvester \cite{Syl} showed that there exists a minimum positive number $f(a_1,\dots,a_m)$, which is now called the Frobenius number, such that
$$
f(a_1,\dots,a_m)+1+ \mathbb{N}\subset {\mathcal J}.
$$
The following
so-called {\em diophantine Frobenius Problem} was raised by F. G. Frobenius
(see \cite{Alf05}):  find $f(a_1,\dots,a_m)$ the largest natural number
that is not
representable as a non-negative integer combination of $a_1, \cdots, a_m$.

Sylvester's result says that a translation of the set $\mathbb{N}$ is contained in $\mathcal J$. Thus the structure of $\mathcal J$ is rather well described by the Frobenius number.
It turned out that the knowledge of $f(a_1,\dots,a_m)$ has been extremely useful to
investigate many different problems. When $m=2$,
 it is well-known that $$f(a_1,a_2)=a_1a_2-a_1-a_2.$$
 For example, if  $a_1=3$ and $a_2=5$, then   $$
  {\mathcal J} =3\mathbb{N}+5\mathbb{N}=\{0,3,5,6,8,9,10,11,\dots\}$$
 and $f(3,5)=7$. However,  for $m\geq 3$, there is no closed form
 for $f(a_1, \cdots, a_m)$. It is proved that the Frobenius problem is NP-hard under
Turing reductions.  The book of Ram\'irez-Alfons\'in \cite{Alf05} (2005)
 is a nice survey on the Frobenius problem.

 Our study of higher dimensional Frobenius problem is motivated by the study of Lipschitz equivalence of Cantor sets.
Lipschitz equivalence preserves many
important properties of a self-similar set.
A survey on   Lipschitz equivalence of Cantor sets can be found in \cite{RRW13}, see also \cite{Lau13}.
 In this area, a fundamental  problem initially raised by Falconer and Marsh \cite{FaMa92}, which is now called the Falconer-Marsh problem, is as follows:
 Assume that   two self-similar sets $E$ and $F$  are dust-like. How  is the Lipschitz equivalence related to the contraction ratios of $E$ and $F$ ?
Falconer and Marsh \cite{FaMa92} established several basic techniques and results in 1992. But there is no  progress until recent works of
Rao, Ruan and Wang \cite{RRW12} (2012) and Xiong and  Xi \cite{XiXi13} (2013).

 Xiong and  Xi \cite{XiXi13}  studied the case
when  $E$ and $F$ have rank $1$ (i.e. contraction ratios are powers of a fixed number)  and discovered that  the problem is closely related to   the class number of the field generated by the ratios.

Rao, Ruan and Wang \cite{RRW12} introduce a Lipschitz invariant described by a so-called \emph{matchable condition}. They solved the problem when both $E$ and $F$ have full rank or both of them
are two-branch self-similar sets.
However, the matchable condition is hard to check in general. The present paper and the sequential paper \cite{RZ14} introduce new techniques to handle the matchable condition.
We associate to each self-similar set a higher dimensional Frobenius problem.
We find that  this   is closely related to the matchable condition.
Thanks to this link, in  \cite{RZ14}, we
solve the Falconer-Marsh problem in the case that
  the contraction ratios of the self-similar sets satisfy a coplanar condition.

 In the following subsections we will describe in some detail our results obtained in the paper.
 Here is a resum\'e.
 Two aspects of the structure of the semi-group $\cal J$ defined by (\ref{equ9}) will be first
 studied: one is the existence and finiteness of maximal saturated cones (Section 2) and the other
 is the growth function which describes how many ways a given vector $z$ in $\cal J$
 can be represented by finite sums of terms from $\{X_1, ..., X_m\}$. We shall prove that
 it is a function which increases exponentially as $\|z\|$ tends to the infinity and that the increasing rate depends on the direction along which $\|z\|$ tends to the infinity (Section 3 and Section 4).
 Thus we obtain the so-called directional growth function.
 An explicit formula is obtained for the directional growth function when the vectors
 $X_1, \cdots, X_m$ are located on a same hyperplane (Section 5). The last two sections
 are devoted to the rigidity. The rigidity means, if two growth functions are equal, then the corresponding
 semi-groups are equal.  Furthermore, the sets of vectors defining the semi-groups are the same.
 These rigidity results are proved under the assumption that the defining vectors are coplanar.

 In the following subsections, we state our results in some details.

\subsection{Maximal saturated cone}
Recall that
$X_1,\dots,X_m$ are $m$ given vectors in the lattice $\mathbb{Z}^s$ ($m\ge 2$, $s\geq 1$) such that
$\langle X_j, \balpha\rangle>0$  for all $j=1,\dots, m$ and for some vector $\balpha \in \mathbb{R}^s$.
For simplicity we use ${ X}$ to denote the set $\{X_1,\dots, X_m\}$. Let
\begin{equation}
{\mathcal L}:= {\mathcal L}_X:= X_1\Z+\cdots+X_m\Z
\end{equation}
denote  the lattice generated by $X_1,\dots, X_m$. Let
\begin{equation}
{\mathbf C}_X:=X_1\R^++\cdots+X_m\R^+
\end{equation}
  denote the convex cone generated by $X_1,\dots, X_m$, where $\R^+$ is the set of non-negative real numbers.
 Clearly, the semi-group $\CJ$ is a subset of the lattice ${\mathcal L}$ so that
 $$
 \forall g\in \CJ, \quad (g+{\mathbf C}_X)\cap {\mathcal J}\subset (g+{\mathbf C}_X)\cap \CL. $$
If
$(g+{\mathbf C}_X)\cap {\mathcal J}=(g+{\mathbf C}_X)\cap \CL$,  the cone
$g+{\mathbf C}_X$ is said to be \emph{saturated}. That means, every lattice point in the cone
$g + {\mathbf C}_X$ is in the semi-group.
Moreover, a saturated cone is said to be \emph{maximal} if it is not a subset of any other saturated cone.
Then we define the \emph{Frobenius set} to be
\begin{equation}\label{Frobenius}
{\mathcal F}:=\{g\in \CJ: ~g+{\mathbf C}_X \text{ is a maximal saturated cone}\}.
\end{equation}
Let us look at  the one-dimensional case considered above with $X_j=a_j$. The cone ${\mathbf C}_X$
 is then equal to $\mathbb{R}^+$. The cone  $g+\R^+$ is saturated  if and only if $g>f(a_1,\dots, a_m)$. Therefore we have ${\mathcal F}=\{f(a_1,\dots, a_m)+1\}$, the singleton consisting of
the smallest natural number such that  all larger natural numbers  are representable as a non-negative  integer combination of $a_1, \cdots, a_m$.

A natural question is ``{\em how many maximal saturated cones are there} ?'' Our answer is

\begin{theorem}\label{thm-MSC} The Frobenius set ${\mathcal F}$ is non-empty and finite.
\end{theorem}

Here is an example where the Frobenius set has two elements.
{\rm Let $s=2$ and $X=\{(3,0), (1,2), (0,3)\}$. Then
${\mathbf C}_X=\R^+ \times \R^+ $ and
$$
\CL=\{(a,b)\in \Z^2; ~a+b \text{ is a multiple of }3\},
$$
$$
\CJ=\{(a,b)\in \N^2; ~a+b \text{ is a multiple of }3 \text{ and } b\neq 1\}.
$$
We find  ${\mathcal F}=\{(1,2),(0,3)\}$. This is shown
in Figure \ref{Fig1}.
}

\begin{figure}
  \includegraphics[width=5cm]{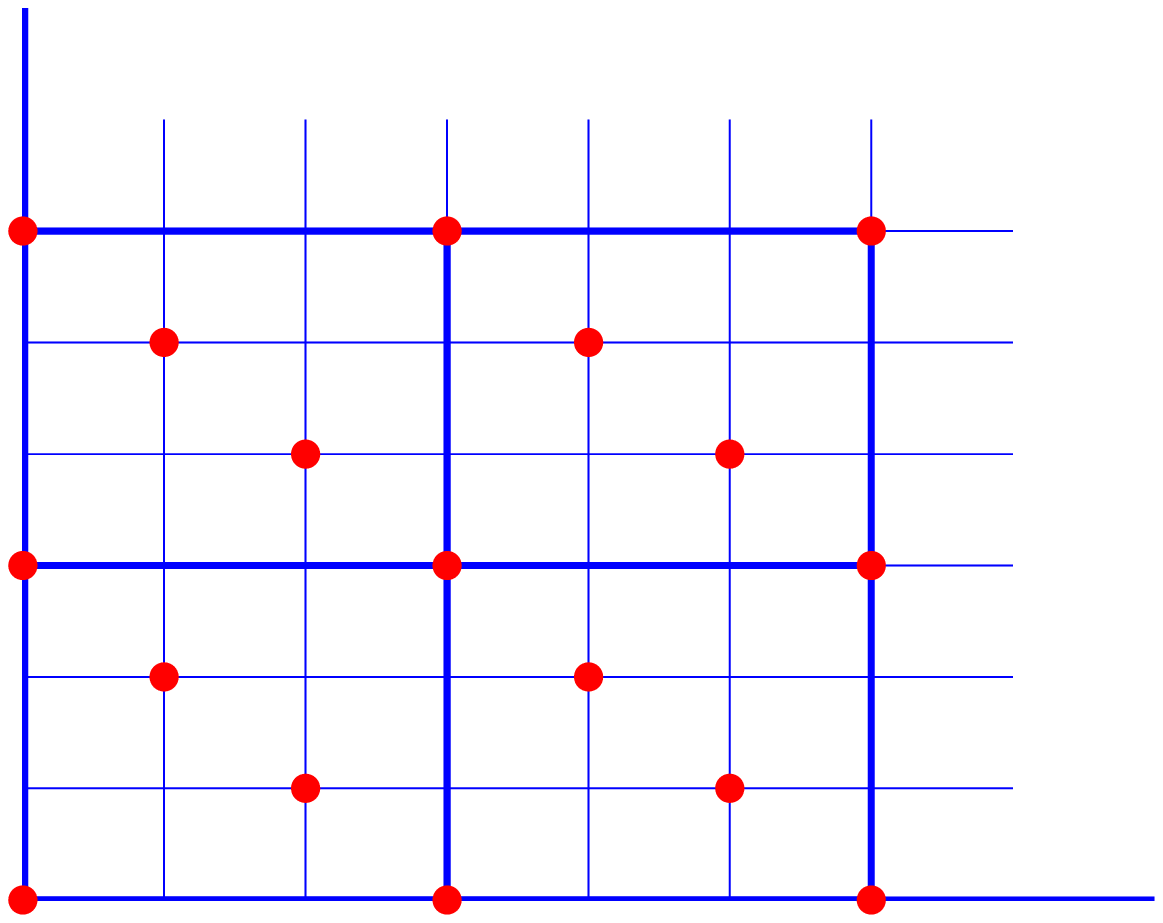}
  \quad  \includegraphics[width=5cm]{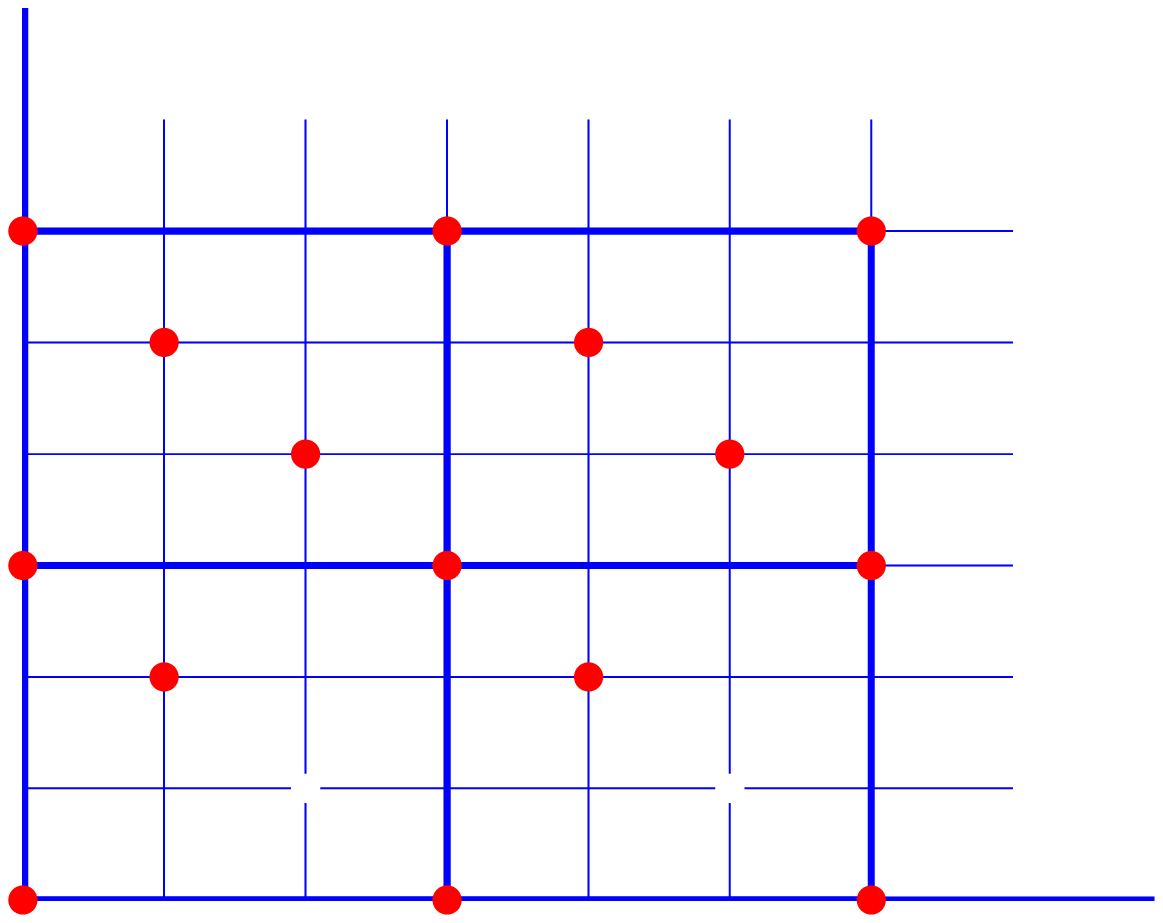}\\
  \caption{Example 1.3: $\CL$ (left) and $\CJ$ (right).}
  \label{Fig1}
\end{figure}

\subsection{Multiplicity of representations and directional growth function}
For any vector $z$ in the semi-group $\cal J$, we are interested in the number of
representations $z = X_{i_1}+ \cdots + X_{i_n}$ where $n\ge 1$ and $i_k$'s
are taken from $\{1, 2, \cdots, m\}$. As we shall see, these numbers reflect some
property of the structure of the semi-group.

Let  $\Sigma_m^*:=\bigcup_{k=0}^{\infty}\{1,2,\dots,m \}^k$ be the set of words over the alphabet
$\{1,2,\dots,m \}$,
which can also be considered as a tree.
For any word $\mathbf{i}=i_1\dots i_n \in \Sigma_m^*,$  define
\begin{equation}\label{kappa}
\kappa (\mathbf{i})=X_{i_1}+\cdots+X_{i_n}.
\end{equation}
We consider $\kappa: \Sigma_m^* \to \mathbb{Z}^s$ as the {\em walk} in $\mathbb{Z}^s$ guided by $X_1,\dots,X_m$ along with the tree $\Sigma_m^*$.
Elements in $\Sigma_m^*$ are also called {\em pathes} of the walk and $\kappa (\mathbf{i})$ is called the {\em visited position}
following the path $\mathbf{i} $.
A point $z\in \mathbb{Z}^s$ is said to be {\em attainable} if $z=\kappa(\mathbf i)$ for some ${\mathbf i}\in\Sigma_m^*.$
 Clearly the set of attainable positions is exactly the semi-group ${\mathcal J}$.
 A second question we ask is ``{\em How many times is an attainable position  visited} ?''

To partially answer this question, for $z\in\mathcal J$, we define
the \emph{multiplicity} of $z$ to be
\begin{equation}\label{equ10}
{\mathbf m}(z):=\# \{\omega\in \Sigma_m^*; \ \kappa(\omega)=z\}.
\end{equation}
We extend the function ${\mathbf m}$ to the convex cone ${\mathbf C}_X$ as follows.
For any point $x\in {\mathbf C}_X$ but not in ${\mathcal J}$, instead of setting ${\mathbf m}(x)=0$,
  we define its multiplicity  to be the multiplicity of the point in $\CJ$ which
is nearest to $x$. More precisely,
\begin{equation}
{\mathbf m}(x):=\min \{{\mathbf m}(z); \ z\in {\mathcal J} \text{ and } \|x-z\|=d(x,{\mathcal J})\},
\end{equation}
where $\|\cdot\|$ denotes the Euclidean norm and $d(x,{\mathcal J}):=\min\{\|x-z\|; ~z\in {\mathcal J}\}$.

In  the one-dimensional case,  the multiplicity ${\mathbf m}$ restricted on $\mathcal J$ satisfies the linear recurrent relation
$$
{\mathbf m}(n)=\sum_{j=1}^m {\mathbf m}(n-a_j).
$$
Hence  we can obtain an explicit formula for
${\mathbf m}(n)$. It is then easy to show that ${\mathbf m}(n)$ is
of the same exponential order as $\beta^n$
where $\beta$ is the largest root of the equation
$$
x^{a_m}-\sum_{j=1}^m  x^{a_m-a_j}=0.
$$
(See, for instance, \cite{Car}).
But in the higher dimensional case, it is hard to obtain an explicit formula for the multiplicity. Nevertheless, we will prove the following exponential growth.

\begin{theorem}\label{thm-growth}
For any unit vector $\theta\in {\mathbf C}_{X}$, the following limit exists
\begin{equation}\label{growth}
\gamma(\theta)=\underset{k \rightarrow  \infty}{\lim}\frac{\log \mathbf{m}(k\theta)}{k}.
\end{equation}
\end{theorem}

We call $\gamma$ the \emph{directional growth function} of the semi-group $\mathcal J$. 
It
 describes the exponential increasing speed of the multiplicity along the direction $\theta$.
We will first prove that the multiplicity function varies slowly in the sense that the quotient
 of $\m(z)$ and $\m(z')$ is of polynomial order of $\|z\|$ if $z'$ and $z$ have a bounded distance (Theorem \ref{Q(x)}). We will then prove that the sequence
$(\log \m(k\theta))_{k\geq 1}$ is subadditive in some weak sense, which is sufficient to ensure the existence of the limit in \eqref{growth}, according to Lemma \ref{lem-subadd} which  strengthens
a classical result on sub-additive sequence.

\begin{figure}
  \includegraphics[width=6cm]{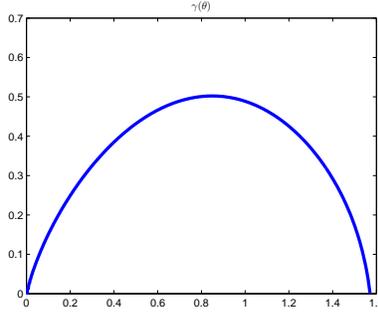}\\
  \caption{The function $\gamma(\theta)$ for $X=\{(3,0),(1,2),(0,3\}$.}
  \label{fig_gamma}
\end{figure}

\subsection{Calculation of $\gamma(\theta)$ when $X_1, \cdots, X_m$ are coplanar}
In general, it is difficult to obtain an explicit formula of $\gamma(\theta)$. We will be confined
to a formula under the condition that  $X_1, \cdots, X_m$ are coplanar.

We say that $X_1,\dots, X_m$ are \emph{coplanar} if they locate on a same hyper-plane, i.e.
there exists a vector $\eta \in \R^s$ such that
\begin{equation}\label{co-plane}
\langle \eta,  X_j \rangle =1, \quad j=1,\dots, m.
\end{equation}
To be more precise, we say $X_1,\dots, X_m$ are \emph{$\eta$-coplanar}.
Let $p=(p_1,\dots, p_m)$ be a probability vector. The {\em entropy} of $p$ is defined as
$$
h(p)=-\sum_{j=1}^m p_j\log p_j.
$$

\begin{theorem}\label{thm-entropy} Suppose that $X_1,\dots, X_m$ are $\eta$-coplanar.
For any  unit vector $\theta$ in  the cone ${\mathbf C}_X$, we have
\begin{equation}\label{entropy}
\gamma(\theta) = {\langle \theta, \eta \rangle } \sup \left \{ h(p);~ p_1X_1+\cdots+p_mX_m=\frac{\theta}{\langle \theta, \eta\rangle} \right \}.
\end{equation}
\end{theorem}

There is another expression involving the following function
     $\log \sum_{j=1}^m e^{\langle t, X_j\rangle},$
         $t \in \mathbb{R}^s,$
which corresponds to the pressure function in the statistic physics (Theorem \ref{thm-Z}).
The above formula (\ref{entropy}) resembles  the conditional variation principle in the analysis of multifractal
analysis (see \cite{FFW}, see also \cite{FF}, \cite{FLP2008}). Actually, the proof of Theorem \ref{thm-entropy} uses the idea of large deviation.
If $X_1,\dots, X_m$ are linearly independent, then the choice of $p$ is unique and
we can easily compute $\gamma(\theta)$.

Here is an example.
 Let $s=2$ and $X=\{(1,0),(0,1)\}$. Then
${\mathbf C}_X=(\R^+)^2, \ {\mathcal J}=\N^2,$  $\CL=\Z^2$ and $\eta=(1,1)$.
Clearly $\m(a,b)=\frac{(a+b)!}{a!b!}.$
For $\theta=(\theta_1,\theta_2)$ with $\theta_1^2+\theta_2^2=1$ and $\theta_1,\theta_2\geq 0$.
Let $p_1=\theta_1/(\theta_1+\theta_2)$, $p_2=\theta_2/(\theta_1+\theta_2)$.
This vector $(p_1, p_2)$ is the unique probability satisfying $p_1X_1 +p_2X_2 = \theta/\langle \eta, \theta\rangle$.   Then
$$
\gamma(\theta)=(\theta_1+\theta_2) h(p_1,p_2).
$$
The unit vector $(\theta_1, \theta_2)$ can be described by the angle $\alpha \in [0, \pi/2]$
such that $\theta_1 = \cos \alpha$. Then
$$
   \gamma(\alpha) = - \cos \alpha \log \frac{\cos \alpha}{\cos \alpha + \sin \alpha}
   - \sin \alpha \log \frac{\sin \alpha}{\cos \alpha + \sin \alpha}.
$$
The maximum is attained at $\pi/4$ and $\gamma(\pi/4)= \sqrt{2} \log 2$.
The formula of $\gamma(\theta)$ in this case can  be directly deduced from the Stirling formula.

Here is another example where
$X = \{(3, 0), (1, 2), (0, 3)\}$.
The graph of the growth function
$\gamma(\theta)$ is shown in Figure 2. 

\subsection{Rigidity results}
Given two sets of vectors
$X=\{X_1,\dots, X_m\}$ and
$Y=\{Y_1,\dots, Y_{m'}\}$ in $\mathbb{Z}^s$. Suppose that they define the same directional growth function, \textit{i.e.}
\begin{equation}\label{equ-XY}
{\mathbf C}_X={\mathbf C}_Y \text{ and } \gamma_X=\gamma_Y.
\end{equation}
 {\em What can we say about $X$ and $Y$}?  In our terminology, Rao, Ruan and Wang \cite{RRW12} proved the following rigidity result.

 \begin{pro} [\cite{RRW12}] \label{RRW}  Suppose    $X=\{X_1,\dots, X_s\}$ and
$Y=\{Y_1,\dots, Y_{s}\}$  are two sets of linearly independent vectors in ${\mathbb Z}^s$.
If they define the same directional growth function, then $X$ is a permutation of $Y$.
 \end{pro}

We will generalize the above result to  the coplanar case. Notice that $X_1, ..., X_m$
are coplanar if $m\le s$ and in particular, linearly independent vectors are coplanar.

\begin{theorem}\label{thm-XY}
 Suppose  $X=\{X_1,\dots, X_m\}$ and
$Y=\{Y_1,\dots, Y_{m'}\}$
are $\eta$-coplanar for some $\eta \in \mathbb{R}^s$  and that they define the same directional growth function. Then
 $m=m'$ and $X$ is a permutation of $Y$.
\end{theorem}

As we shall see, the proof Theorem \ref{thm-XY} is much more difficult than that of Proposition \ref{RRW}. We can still consider
two coplanar sets of vectors which are respectively located on two different hyper-planes.

 Let $X^{(p)}=\left ( \kappa ({\mathbf i})\right )_{{\mathbf i}\in \{1,\dots,m\}^p}$, which is called the $p$-th \emph{iteration} of $X$, where
$\kappa$ is defined in \eqref{kappa}.
For example, the second iteration of $X=\{(1,0),(0,1)\}$ is $\{(2,0),(1,1),(1,1),(0,2)\}$.
Using techniques of algebraic plane curve, we prove that

\begin{theorem}\label{thm-XY-new}
 Suppose  $X=\{X_1,\dots, X_m\}$ is $\eta$-coplanar and
$Y=\{Y_1,\dots, Y_{m'}\}$ is $\eta'$-coplanar, and suppose $X$ and $Y$
  define the same directional growth function. Then
 $\eta=c\eta'$ for some $c>0$  and there exists two integers $n,n'\geq 1$ such that
 the $n$-th iteration of $X$ is a permutation of the $n'$-th iteration of $Y$.
\end{theorem}


\subsection{Relation to the Lipschitz equivalence of Cantor sets}
Let $ \boldsymbol{\rho}=(\rho_{1},\cdot\cdot\cdot,\rho_{m})$ be a vector such that
$ \rho_j\in (0,1)$ for all $j=1, 2, \cdots, m$. An example is the  set of contraction ratios in a contractive self-similar iterated function system. Let $\langle\brho\rangle$ (resp.  $\langle\brho\rangle_+$) denote the subgroup (resp. semi-group) of
 $( \mathbb{R}^{+} ,\times)~$ generated by $ \rho_1,\dots,\rho_m$. Such semi-groups  $\langle\brho\rangle_+$ play a crucial role
in the discussion of the Lipschitz equivalence  of self-similar Cantor sets (\cite{FaMa92}).

 A \emph{pseudo-basis} of $\langle \boldsymbol{\rho}\rangle$ is a set of numbers
 $\boldsymbol{\lambda}=(\lambda_{1},\cdots,\lambda_{s})$ such that
 $\langle \boldsymbol{\rho}\rangle \subset \langle  \boldsymbol{\lambda} \rangle$
 and $s$ is the rank of $\langle \boldsymbol{\rho}\rangle$, i.e. the cardinality of a basis of
 $\langle \boldsymbol{\rho}\rangle$.
 The multiplicative group $\langle \blambda\rangle$ is isomorphic to the additive group $(\mathbb{Z}^s,+)$ and an isomorphism is defined by $\text{exp}_{\blambda}: \mathbb{Z}^s\rightarrow \langle \blambda\rangle$ where
\begin{equation*}
\forall X=(X^1,\dots,X^s)\in \mathbb{Z}^s, \qquad X \mapsto
\boldsymbol{\lambda}^X:=\prod_{i=1}^s{\lambda_i}^{X^i}.
\end{equation*}
The inverse map $\log_{\boldsymbol{\lambda}}: \langle \blambda\rangle\rightarrow \mathbb{Z}^s$
of the isomorphism is then defined  by
$
 \log_{\boldsymbol{\lambda}}x=X
$ for $x=\boldsymbol{\lambda}^X \in \langle \blambda\rangle$. Let
\begin{equation*}
X_i=\log_{\boldsymbol{\lambda}}\rho_i, \quad (i=1, \dots, m).
\end{equation*}
Then the multiplicative semi-group  $\langle\brho\rangle_+$ is isomorphic to the additive semi-group
$\mathcal{J}(\brho) = X_1 \mathbb{N} + \cdots + X_m \mathbb{N}$.

Given two Cantor sets generated by self-similar iterated function systems. We fix
a common pseudo-basis $\boldsymbol{\lambda}$ for both sets of contractions, denoted
$\boldsymbol{\rho}$ and $\boldsymbol{\rho'}$.
Such a pseudo-basis does exist when the two Cantor sets are Lipschitz equivalent (\cite{FaMa92}).
 Under the assumption
that both sets of vectors $\log_{\boldsymbol{\lambda}} \boldsymbol{\rho}$
and $\log_{\boldsymbol{\lambda}} \boldsymbol{\rho'}$ are coplanar,
it will be proved that two such Cantor sets are Lipschitz equivalent if and only if
the situation described in the rigidity Theorem \ref{thm-XY-new} takes place
 (see \cite{RZ14}).

\section{\textbf{Maximal saturated cones}}

We assume that $X_1,\dots, X_m\in \Z^s$ locate on a same half-plane, that is, there exists
a vector $\boldsymbol{\alpha} \in \mathbb{R}^s$ such that $\langle X_j, \boldsymbol{ \alpha} \rangle>0$ for
$j=1,\dots,m$.
%
Recall that  ${\mathbf C}_X=X_1\R^+\cdots+X_m\R^+$ and
  ${\cal L}=X_1\Z+\cdots+X_m\Z$  are respectively the cone and  the lattice generated by $X_1,\dots, X_m$.

 Remark that we can work in a little more general setting.  Let $E$
  be a Euclidean space and $L$ be a lattice of full rank in $E$.  Given $m$ non-zero points $X_1, ..., X_m$
   of the lattice, we can consider the generated semi-group $\mathbb{N} X_1 + \cdots + \mathbb{N}X_m$.  In other word, there is no need to work with
   the orthogonal lattice $\mathbb{Z}^d$. All the results we will present remain true in this setting.

\begin{theorem}\label{thm-2.1}
There exists $g\in {\mathcal J}$ such that $g+{\mathbf C}_X$ is a  saturated cone.
\end{theorem}

\begin{proof}
Set
$\Omega =\left \{\sum_{j=1}^m c_j X_j;~c_j\in [0,1)\right \}$, considered as basic domain. Then
every
$x=c_1X_1+\cdots+c_mX_m \in {\mathbf C}_X$ can be written as
\begin{equation}\label{eq-dense}
x=y+\omega
\end{equation}
where $y=\lfloor c_1 \rfloor X_1+\cdots+\lfloor c_m \rfloor X_m$ is in the semi-group ${\mathcal J}$ and $\omega\in \Omega$.
Let
$$\Omega^*=\Omega\cap \CL.$$
Since $\Omega$ is bounded,  $\Omega^*$ is a finite set.
Then there exists an integer $M$ such that for
every $\omega\in \Omega^*$, there exists $a_1,\dots,a_m \in \Z$ such that
$$
\omega=\sum_{j=1}^m a_jX_j, \qquad |a_j|\le M.
$$
Put $g=M(X_1+\cdots+X_m).$ We claim that the cone $g+{\mathbf C}_X$ is saturated.

In fact,
  let $z\in (g+{\mathbf C}_X)\cap {\mathcal L}$. Since $z-g\in {\mathbf C}_X$, as we have just seen, we can write
$$z-g=x+\omega, \qquad  \mbox{\rm with}\ x\in \CJ, \omega\in \Omega.$$
Since $z$, $g$ and $x$ all belong to $\CL$,
so does $\omega$. Hence $\omega \in \Omega^*$.
By the definition of $g$, it is clear that $g+\omega\in \CJ$. Thus
 $z=x+(g+\omega)\in \CJ$.
\end{proof}




\medskip

\noindent \textbf{Proof of Theorem \ref{thm-MSC}.}
The existence of saturated cones is confirmed by Theorem \ref{thm-2.1}.
For a given saturated cone, there is at most a finite number of saturated cones which contain the given one. This finiteness implies the existence of maximum saturated cone.

In the following, we show that the number of maximal saturated cones is finite by contradiction.
Suppose on the contrary that the Frobenius set ${\mathcal F}$ is infinite.
For any $g\in {\mathcal F}$, choose a path $\omega$ such that $g=\kappa (\omega)$ and set
 $u(g)=(u_1,\dots,u_m)$ where $u_j$ counts the number of the symbol $j$ in $\omega$.
(We remark that the choice of $\omega$ is not unique.)

By the definition of maximal saturated cone, for any $g_1,g_2\in {\mathcal F}$,
both $g_1-g_2$ and $g_2-g_1$ do not belong to $\CJ$. Hence,
$u(g_1)$ and $u(g_2)$ are not comparable, \textit{i.e.}, both $u(g_1)-u(g_2)$ and $u(g_2)-u(g_1)$
are not non-negative vectors. However, since the set
$\{u(g);~g\in {\mathcal F}\}\subset \N^m$ is infinite, there must exist two comparable elements.
This contradiction proves the theorem.
$\Box$

\medskip

As a direct consequence of \eqref{eq-dense}, we have 

\begin{lemma}\label{dense}
The  set ${\mathcal J}$ is \emph{relatively dense} in   ${\mathbf C}_X$,  that is, there exists a constant $R_0>0$ such that
$d(x, {\mathcal J})<R_0$ for every $x\in {\mathbf C}_X.$
\end{lemma}

\section{\textbf{Variation of multiplicity function}}
We are now going to prove that the multiplicity $\m(z)$
has an exponential increasing rate as $z$ tends to the infinity along each direction in the cone
${\mathbf C}_X$.

First of all, we give another expression for the multiplicity function $\mathbf{m}.$
For $z\in \CJ,$  define
\begin{equation}
A(z):= \left \{ (u_{1},\dots, u_{m})\in \mathbb{ N}^{m};~\sum_{j=1}^{m}u_{j}X_{j}=z \right \}.
\end{equation}
The cardinality  $\#A(z)$ is the number of ways that $z$ can be represented as linear combination of $X_1, \cdots, X_m$ with non-negative integer coefficients.
In one dimensional case, the function
$z \mapsto \#A(z)$
 is the \emph{denumerant function} introduced by Sylvester \cite{Syl}.
 We have the following expression for $\m(z)$:
\begin{equation}\label{m_z}
\mathbf{m}(z)=\underset{u\in A(z)}{\sum}\frac{ |u|! }{u!}
\end{equation}
 where
$|u|=u_1+\cdots+u_m$ and $ u!= u_1!\cdots u_m!$
for a multi-index  $u= (u_{1},\dots, u_{m} )\in \mathbb{N}^m $.

Assume $X=\{(1,0),(0,1)\}$. Then $\CJ=\N^2$ and for $(a,b)\in \N^2$, we have
$$
\m(a,b)=\frac{(a+b)!}{a!b!}.
$$
Hence if the distance of two points $(a,b)$ and $(a',b')$ are bounded by a constant, we see that $\m(a,b)/\m(a',b')$
is controlled by a polynomial of $\sqrt{a^2 +b^2}$. As we shall show in Theorem \ref{Q(x)},  that is the case in general.



For $ z=(z^1,\dots,z^s)\in \mathbb{ R}^{s}$,    define
$\|z\|_\infty:=\max \{ |z^1|,\dots,  |z^s|\}.$
 Let $d_H$ denote the Hausdorff metric, that is, if $A$ and $B$ are two subsets of $\R^s$,
 then
 $$
 d_H(A,B)=\max\left \{\sup_{a\in A}d(a,B), \sup_{b\in B}d(b,A)\right \}.
 $$

 Recall that $\langle X_j, \boldsymbol{\alpha} \rangle>0$ for all $j\in \{1,\dots, m\}$.  Set
\begin{equation}\label{delta}
 \delta=\underset{1\leq j \leq m}{\min}\frac{\langle X_j,\boldsymbol{\alpha} \rangle}{\|\boldsymbol{\alpha}\|}.
\end{equation}

\begin{lemma}\label{lem-3.1} The set $A(z)$ is contained in a $\|\cdot\|_\infty$-ball of radius
 $\|z\|/\delta$. In other word,
$\|u\|_\infty\leq \|z\|/\delta$ for all $u\in A(z)$.
\end{lemma}

\begin{proof} Let $u \in A(z)$ so that $z = \sum_{j=1}^m u_jX_j$. Using Schwartz inequality,
 we obtain
$$\displaystyle  \|z\| \geq \frac{\langle z,\boldsymbol{\alpha}\rangle}{\|\boldsymbol{\alpha}\|}
 = \sum_{j=1}^m u_j  \frac{\langle X_j,\boldsymbol{\alpha} \rangle}{\|\boldsymbol{\alpha}\|}
 \geq \delta\|u\|_{\infty}.$$
 The lemma follows.
\end{proof}

The following theorem plays a crucial  r\^ole in our argument.

\begin{thm}\label{d_H}
Let $C_0\ge 1$ be an integer. Then there exists an integer $M>0$ such that
$$
 d_H(A(z), A(z'))<M,
$$
provided that  $z,z'\in {\mathcal J}$ and $\|z-z'\| \le C_0$.
\end{thm}

\begin{proof} For $u,v\in \mathbb{R}^m,$  we define the order $u\preceq v $ if $v-u$ is a non-negative vector.
Pick any  $u=(u_1,\dots,u_m)\in A(z).$ We claim that there exists $z^*\in \CJ$
such that

\ \ (i) there exists $u^*\in A(z^*)$ such that  $u^*\preceq u$ (this  implies that $z-z^*\in \CJ$);

\ (ii) $z'-z^*\in \CJ$ ;

(iii) $\|z-z^*\|\le M'$, where $M'$ is a constant  depending only on $X_1,\dots, X_m$ and $C_0$.

Notice that $z=z^* + (z-z^*)$ and $z'=z^* + (z'-z^*)$ with $z^*, z-z^*, z'-z^*$ belonging to $\CJ$.
 Roughly speaking,  the point $z^*$ is not far away from $z$ and one
 can walk from $0$ to $z^*$. From there one can walk to  $z$ as well as to $z'$. (See Figure \ref{KKK} left.)

\begin{figure}[h]
  \includegraphics[height=4cm]{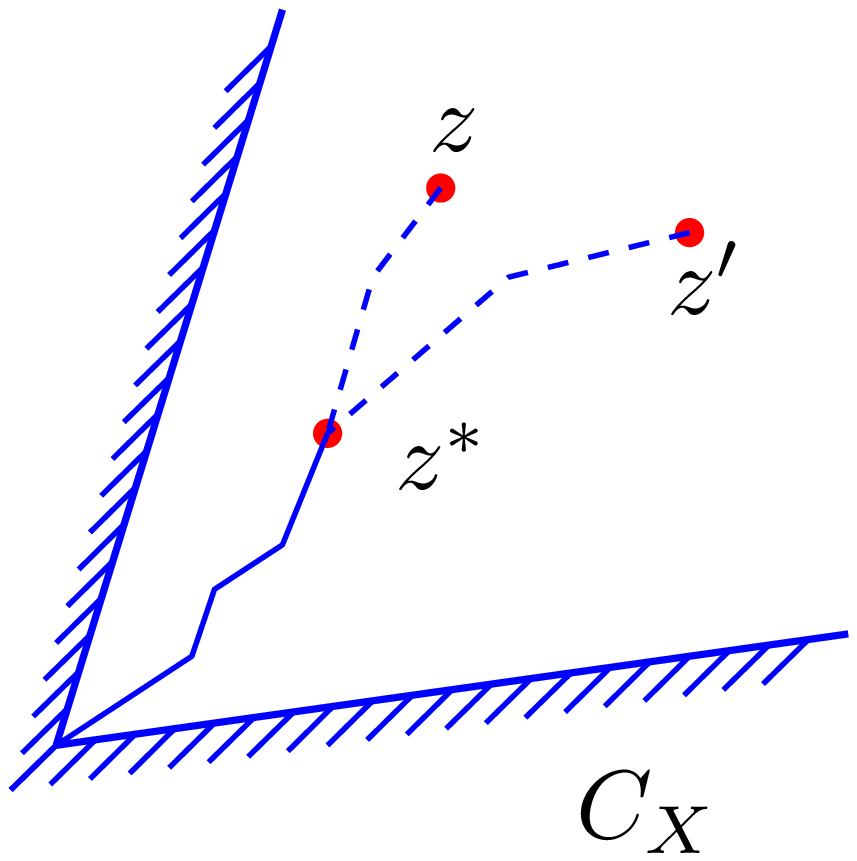}\quad
  \includegraphics[height=4cm]{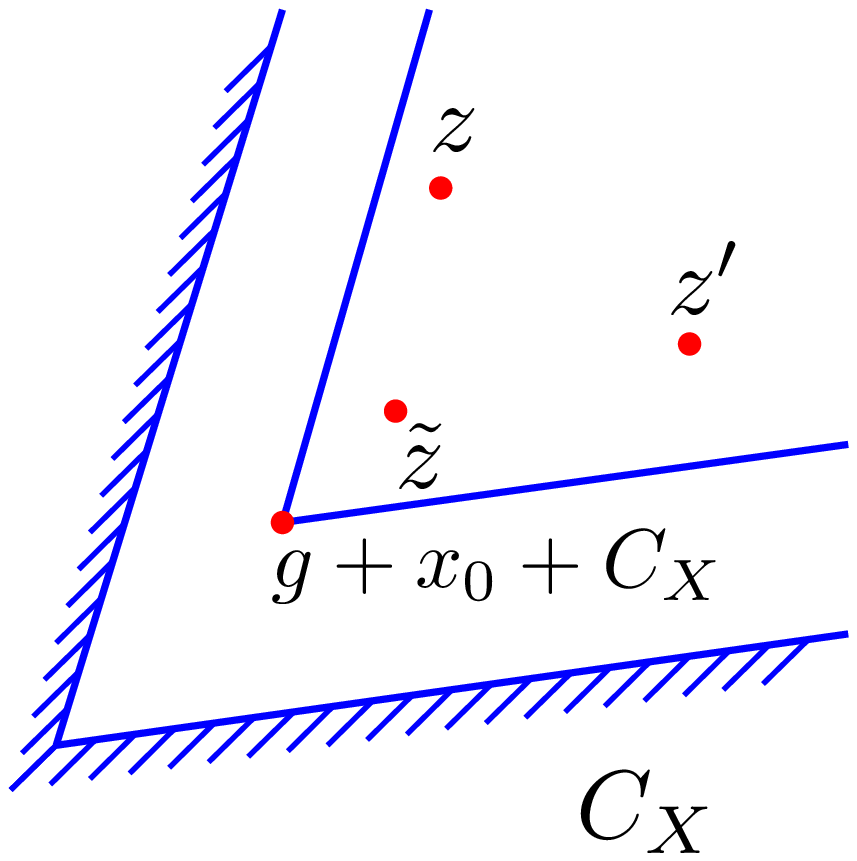}
  \caption{}\label{KKK}
\end{figure}

Suppose the claim is proved. Take any $v'=(v_1',\dots,v_m')\in A(z'-z^*)$ and set $v=u^*+v'$.
 Then $v\in A(z')$ and
 \begin{eqnarray*}
 \|u-v\|_\infty &\leq & \|u-u^*\|_\infty+\|v-u^*\|_\infty\\
 &\leq & (\|z-z^*\|+\|z'-z^*\|)/\delta \quad (\text{By Lemma \ref{lem-3.1}})\\
 &\leq & (2M'+C_0)/\delta.
 \end{eqnarray*}
 We have thus proved the theorem by choosing $M=\sqrt{m}(2M'+C_0)/\delta$.

Now it suffices to prove  the claim.  Take $x_0\in \CJ$ such that $B(x_0,C_0)\subset {\mathbf C}_X,$
where
$B(x,r)$ denotes the ball with center $x$ and radius $r$.
Take $g\in \CJ$ such that  $g+{\mathbf C}_X$ is a saturated cone.  We will prove the claim by distinguishing two cases according to
whether $z\in g+x_0+{\mathbf C}_X$.

\medskip

\textit{Case 1:} $z\in g+x_0+{\mathbf C}_X.$

A point $\tilde{z}=\sum_{j=1}^m \tilde{u}_jX_j$ is called a
\emph{first entering position}
if $\tilde{u}=(\tilde{u}_1,\dots,\tilde{u}_m)$ is a minimal vector (w.r.t. the order $\preceq$) such that $\tilde{z}\in g+x_0+{\mathbf C}_X$ and
$\tilde{u}\preceq u$.
Such $\tilde{z}$ do exist, but may not be unique (See Figure \ref{KKK} right).
We fix such a point $\tilde z$ and set  $z^*=z-\tilde{z}$.
Since $\tilde{u}\preceq u$, we have $z^* \in \CJ$.
We are going to check that (i)-(iii) hold for this $z^*$.

(i) holds since $\tilde{u}\preceq u$.

From $z-z^*=\tilde{z}\in g+x_0+{\mathbf C}_X$ we deduce that
 $z-z^*-x_0 \in g+ {\mathbf C}_X$. This, together with
  $\|z-z'\|\leq C_0$, implies
$$ z'-z^*\in B(z,C_0)-z^*=z-z^*-x_0+B(x_0,C_0)\subset (z-z^*-x_0)+{\mathbf C}_X
\subset g+{\mathbf C}_X.
$$
Hence $z'-z^*\in (g+{\mathbf C}_X)\cap \CL$, and so that $z'-z^*\in {\mathcal J}$
 by the saturation property of $g+C_X$.
This proves the property (ii).

Recall that $\tilde{z}\in g+x_0+{\mathbf C}_X$. The minimality of  $\tilde{z}$  implies
$$ \tilde{z}\in (g+x_0)+\left \{\sum_{j=1}^mc_jX_j;0\leq c_j<1\right \}.$$
It follows that
$$
\|z-z^*\|=\|\tilde{z}\|\leq\|g+x_0\|+ \sum_{j=1}^m\|X_j\|:=M',
$$
and hence (iii) holds.

\medskip

\textit{Case 2.}   $z\not \in g+x_0+{\mathbf C}_X.$

In this case, we could say that $z$ is close to a face of the cone ${\mathbf C}_X$. Indeed, the boundary of ${\mathbf C}_X$  can be written as
 $$\partial {\mathbf C}_X=\bigcup_{j=1}^N {\cal D}_j,$$
where ${\cal D}_j$ are faces of  ${\mathbf C}_X$, which are cones of  dimension $s-1$.
Let $\vec{n}_j$ be the unit vector perpendicular to ${\cal D}_j$  and pointing to the half-space containing ${\mathbf C}_X$, that is to say,
$\langle\vec{n}_j,x\rangle\geq 0$ for all $x\in {\mathbf C}_X.$
We note that  $x\in {\mathbf C}_X$ if and only if $\langle\vec{n}_j,x\rangle\geq 0$ for all $j\in \{1,\dots,  N\}.$

Since $z\not \in g+x_0+{\mathbf C}_X,$
  there exists an integer $j_0 \in \{1,\dots,N\}$ such that $\langle\vec{n}_{j_0},z-(g+x_0)\rangle<0 $, namely,
\begin{equation}\label{tube}
  \langle z,\vec{n}_{j_0}\rangle<\langle g+x_0,\vec{n}_{j_0}\rangle,
 \end{equation}
which means that  $z$ is closer to the face ${\mathcal D}_{j_0}$ than $g+x_0$. 
Let $\Omega_{j_0}=\{i;~\langle X_i,\vec{n}_{j_0}\rangle=0\}$.
Take any $u=(u_1,\dots, u_m)\in A(z)$ and set
$$z_1=\underset{i\in \Omega_{j_0}}{\sum}u_iX_i.$$
Clearly, $z_1\in {\cal D_{j_0}}$. Denote
$ \delta'=\min\{\langle X_i,\vec{n}_{j_0}\rangle;i\in\{1,\dots, m\} \setminus\Omega_{j_0}\}.$
Since $\langle X_j,\vec{n}_{j_0}\rangle\geq 0$ for all $j \in \{1,2, \cdots, m\}$, we  deduce that
for all $i \in\{1,\dots,m\}\setminus\Omega_{j_0}$ we have
$$
u_i \leq \frac{\langle z,\vec{n}_{j_0}\rangle}{\delta'}\leq \frac{\langle g+x_0,\vec{n}_{j_0}\rangle}{\delta'}:=M_1.
$$
It follows that
\begin{equation}\label{bd-u}
\|z-z_1\|\leq m M_1 \underset{1\leq j \leq m}{\max}\|X_j\|.
\end{equation}

\begin{figure}[h]
  \includegraphics[height=4cm]{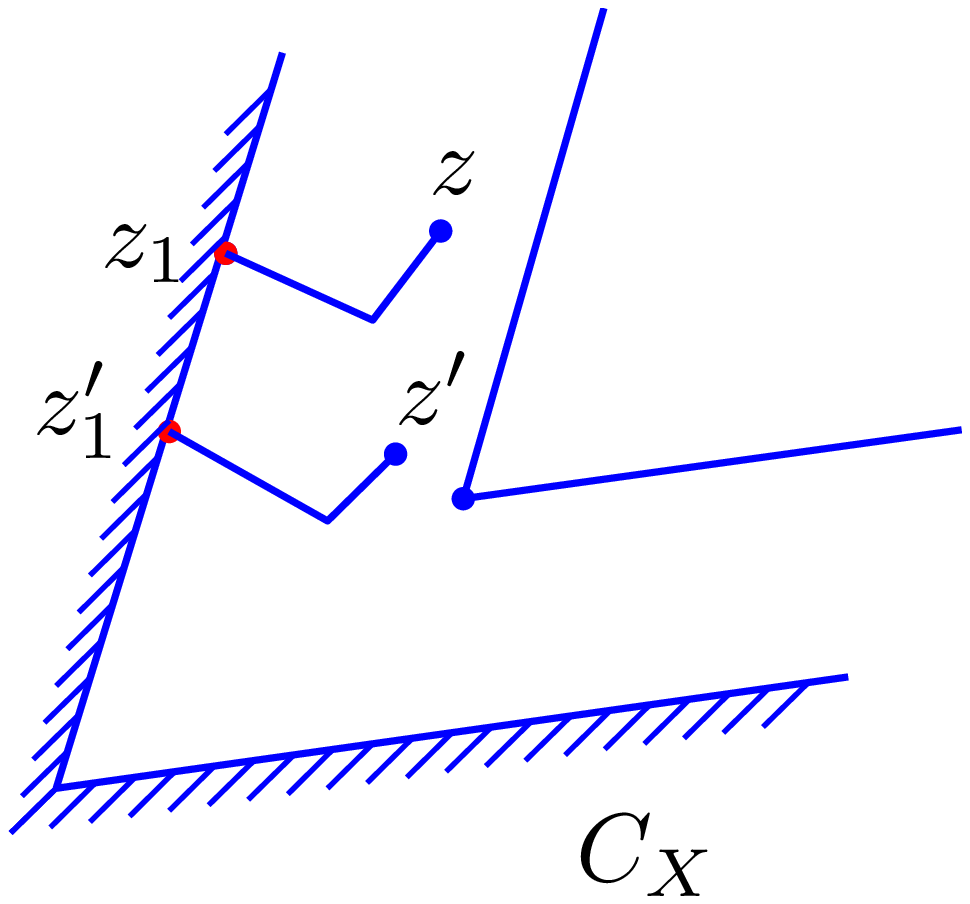}\\
  \caption{}\label{LLL}
\end{figure}

Take any $v=(v_1,\dots,v_m)\in A(z')$  and set $z'_1=\underset{i\in \Omega_{j_0}}{\sum}v_iX_i.$
Then $z'_1\in {\cal D_{j_0}}$. A similar argument as above shows that
$
v_i \leq M_1+C_0/\delta':= M_2,\  \text{for all }\  i\in\{1,\dots,m\}\setminus\Omega_{j_0},
$ and
\begin{equation}\label{bd-v}
 \|z'-z'_1\| \leq mM_2 \underset{1\leq j \leq m}{\max}\|X_j\|.
 \end{equation}

According to \eqref{bd-u} and \eqref{bd-v},  to prove the claim holds for $z$ and $z'$,
we only need  prove that the claim hold for $z_1$ and $z_1'$. That is say, suppose we can walk from the origin to a suitable point
$z_1^*$ and then walk from $z_1^*$ to $z_1$ as well as to $z_1'$. These walks remains in
 the lower dimensional cone ${\cal D}_{j_0}$. Of cause, we can finally walk  to $z$ and to $z'$.
 Observe that 
  $$\|z_1-z'_1\|  <m(M_1+M_2)\underset{1\leq j \leq m}{\max}\|X_j\|+C_0.$$
  The claim can thus be proved by induction on the dimension of the cone.
\end{proof}

The following theorem asserts that $\m(z)$ varies not so rapidly.
Theorem \ref{d_H} will be useful for its proof.

\begin{thm}\label{Q(x)}
Let $C_0\ge 1$ be an integer.  There exists a polynomial $Q(x)$ with positive coefficients such that
$$
 \frac{1}{Q(\|z\|)}\leq\frac{{\mathbf m}(z)}{{\mathbf m}({z^{'}})}\leq Q(\|z\||)
$$
provided that $z,z'\in {\mathcal J}$ and $\|z-z'\| \le C_{0}.$
\end{thm}

\begin{proof}  Let $M$ be the constant in  Theorem \ref{d_H}, which depends on $C_0$.
There exists a  polynomial $G(x)$  (depending on $m$ and $M$) with positive coefficients such that
$$\frac{|v|!}{|u|!}\cdot \frac{u!}{v!}\leq G(\|u\|)$$
holds for all $u, v \in \mathbb{N}^m $ such that $\|u-v\|\le M$.
In other word,
$$
\frac{|u|!}{u!} \geq   \frac{1}{G(\|u\|)} \frac{|v|!}{v!}.
$$
We can take $G(x) = x^{2mM}$.
Denoting by $N_0$  the number of integer points in the ball $B(0,M)$, we have
 \begin{equation}\label{Short}
 N_0 \frac{|u|!}{u!} \geq
\frac{1}{G(\|u\|)} \underset{v:\|u-v\|\le M}{\sum}\frac{|v|!}{v!}.
\end{equation}
Summing up both sides of \eqref{Short} over $u\in A(z)$, we obtain
\begin{equation}\label{Long}
  N_{0} \underset{u\in A(z)}{\sum}{\frac{|u|!}{u!}}
   \geq
  \frac{1}{G(\|u^*\|)} \underset{u\in A(z)} \sum \left (\underset{v:\|u-v\|\le M}{ \sum}\frac{|v|!}{v!} \right )
     \geq
  \frac{1}{G(\|u^*\|)}\underset{v \in A(z') }{ \sum}\frac{|v|!}{v!},
\end{equation}
where $u^*$ is a point in $A(z)$ such that $\|u^*\|$ attains the maximum.
The last inequality holds because
 $$ A(z') \subset \underset{u\in A(z)}{\bigcup} \{v\in {\mathbb N}^s;~ \|u-v\|\le M\}$$
by Theorem \ref{d_H}.
By recalling the expression (\ref{m_z}) for $\mathbf{m}(z)$,
 we see that (\ref{Long}) is nothing but
$$\frac{\mathbf{m}(z)}{\mathbf{m}(z')}\geq \frac{1}{N_0G(\|u^*\|)}.$$
Since
$\|u^*\|\leq m \|u^*\|_\infty \leq m \|z\|/\delta$ (by Lemma \ref{lem-3.1}), where $\delta$ is defined by \eqref{delta}, we have
$$
\frac{1}{N_0G(m \|z\|/\delta)} \leq \frac{\mathbf{m}(z)}{\mathbf{m}(z')}\leq
N_0G(m \|z'\|/\delta)
 $$
 where the second inequality is obtained from the first one by symmetry.
 Notice that $\|z'\|\leq \|z\| +C_0$. We have thus proved the  theorem with
$Q(x)=N_0G\left (\delta^{-1}{m(x+C_0)}\right )$.
\end{proof}

\section{\textbf{Existence of directional growth}}
In this section we prove the existence of the limit defining the directional growth $\gamma(\theta)$.
 Recall that  the multiplicity of a point in ${\mathbf C}_X$ is defined as
$$
{\mathbf m}(x)=\min \{{\mathbf m}(z): \ z\in {\mathcal J} \text{ and } \|x-z\|=d(x,{\mathcal J})\}.
$$

\medskip

\noindent \textbf{Proof of Theorem \ref{thm-growth}.} Fix a unit vector $\theta$ in ${\mathbf C}_X$.
 Let us denote $z_{k}$ to be the  point in $\CJ$ such that
 $$
  \m(k\theta)=\m(z_{k}), \quad d(k\theta,z_{k})=d(k\theta,\CJ).
 $$
 By the relative density of ${\mathcal J}$ (Lemma \ref{dense}), there exists a constant $R_0>0$ such that
\begin{equation}\label{R-dense}
\|z_{k}-k\theta\|<R_0
\end{equation}
for all $k$ and all $\theta$.
 By the definition of multiplicity, it is obvious that  $\mathbf{m}(z+z')\geq \mathbf{m}(z)\mathbf{m}(z')$ for any $z,z'\in {\mathcal J}$. In particular, for any $n,p\geq 1$,
\begin{equation}\label{sub-cross}
\mathbf{m}(z_{n}+z_{p})\geq \mathbf{m}(z_{n})\mathbf{m}(z_{p}).
\end{equation}
As consequence of (\ref{R-dense}), we have
$
\|z_{n+p}- (z_{n}+z_{p})\| \leq 3R_0.
$
Hence, by Theorem \ref{Q(x)},
there exists a polynomials $Q(x)$ with positive coefficients such that
$$
\mathbf{m}(z_{n+p})\geq \frac{1}{Q(\|z_{n+p}\|)}\mathbf{m}(z_{n}+z_{p})
\geq \frac{1}{Q(\|z_{n+p}\|)}\mathbf{m}(z_{n})\m(z_{p}).
$$
Using the fact $\mathbf{m}(k\theta)=\mathbf{m}(z_{k})$ for all $k$, we get then
$$\m ((n+p)\theta)\geq
\frac{1}{\widetilde Q (n+p)} \m (n\theta)\m (p\theta),
 $$
 where $\widetilde  Q (x)=Q (x+R_0)$. Here we used the observation that
$\|z_{n+p}\|\leq (n+p)+R_0$ and the fact that
$Q(x)$ has positive coefficients to get $Q(\|z_{n+p}\|)\le \widetilde Q (n+p)$.
Observe that $\log \widetilde Q (n+p)\leq c \log (n+p)$ for some constant $c>0$.
We can finish the proof by using the next lemma to $b_n = \log \m (n\theta)$.
\ $\Box$

\medskip
It is well known that if a sequence $b_n$ is sub-additive, \textit{i.e.}, $b_n+b_m\leq b_{n+m}$, then the limit
$\lim b_n/n$ exists. The following lemma strengthen this result.

\begin{lemma}\label{lem-subadd}
Let $c>0$ be a constant. If $\{b_n\}$ be a sequence in $\mathbb{R}^+$ such that
 \begin{equation}\label{sub-add}
 b_n+b_m \leq b_{n+m}+c\log (m+n),\quad \text{ for all } m,n\geq 1,
 \end{equation}
 then the limit $\underset{n \rightarrow  \infty}{\lim} b_n/n$ exists.
\end{lemma}

\begin{proof} Without loss of generality, we can assume $c=1$, otherwise we consider $c^{-1}b_n$
instead of $b_n$.
Fix a positive integer $m$.
By  (\ref{sub-add}) we have
\begin{eqnarray}\label{ind}
\forall s\geq 1, \qquad
b_{2^sm} \geq 2b_{2^{s-1}m} - \log(2^sm).
\end{eqnarray}
Fix a positive integer $l \geq 1$. For $ 1\leq s \leq l$, as consequence of  (\ref{ind})  we get
\begin{eqnarray}\label{ind1}
b_{2^lm}\geq
 2^lb_m-[2^{l-1}\log(2m)+2^{l-2}\log(4m)+\cdots+\log(2^lm)]
 =
 2^lb_m-g(m,l)
\end{eqnarray}
where $$
 g(m,l)=(2^{l+1}-2-l)\log 2+(2^l-1)\log m.
 $$

For any $n\in \mathbb{N},$ we write $n=Lm+r$ with $0\leq r<m$.
By (\ref{sub-add}), we have
$$b_n\geq b_{Lm} +b_r- \log n.$$
In order to estimate $ b_{Lm}$, we
use $L$'s dyadic expansion $L=2^{l_1}+2^{l_2}+\dots+2^{l_k}$ where
$ l_1>l_2>\dots >l_k\geq 0$.
Obviously  $ \log_2 L-1 \leq l_1 \leq \log_2 L$.
Using (\ref{sub-add}) then (\ref{ind1}), we obtain
\begin{eqnarray*}
b_{Lm}
&\geq&
b_{2^{l_1}m}+b_{(2^{l_2}+\dots+2^{l_k})m}-\log((2^{l_1}+2^{l_2}+\dots+2^{l_k})m)   \\
&\geq&
2^{l_1}b_m-g(m,l_1)-\log((2^{l_1}+2^{l_2}+\dots+2^{l_k})m)+b_{(2^{l_2}+\dots+2^{l_k})m}.
\end{eqnarray*}
Hence, by induction, we have
$$
b_{Lm} \geq \sum_{j=1}^{k-1}\left(2^{l_j}b_m-g(m,l_j)-\log((2^{l_j}+\cdots+2^{l_k})m)\right)+2^{l_k}b_m-g(m,l_k).
$$
Note that $\sum_{j=1}^k2^{l_j}=L$ and
$$\sum_{j=1}^k g(m,l_j)\leq \sum_{j=1}^k 2^{l_j}\log 4m \leq L \log (4m) ,$$
 $$ \sum_{j=1}^{k-1}\log((2^{l_j}+\cdots+2^{l_k})m)\leq l_1\log (Lm) \leq \log_2 n \log n.$$
Hence we have
$$ b_{Lm}\geq Lb_m-L\log 4m- \log_2 n \log n.$$
Dividing both sides by $n$, taking the liminf and using the fact $L/n\to m$, we have
\begin{equation}\label{b_n}
\underset{n\rightarrow \infty}{\liminf}\ \frac{b_n}{n}\geq  \frac{b_m-\log 4m}{m}.
\end{equation}
Then taking the limsup as $m\to \infty$ finishes the proof.
\end{proof}

\section{\textbf{Principle of Maximal entropy under linear constraints}}

Let $\Delta_m$ be the simplex of all probability measures $p=(p_1,\dots, p_m)$. The entropy function
 $h(p)$ is defined on $\Delta_m$ as follows
$$
h(p)=-\sum_{j=1}^m p_j\log p_j.
$$
Let $X_1, \cdots, X_m$ be $m$ given vectors in a Euclidean space $\mathbb{R}^d$. We will consider
the maximum of $h(p)$ under the constraints
\begin{equation}\label{constraints}
    \sum_{j=1}^m p_j=1, \qquad \sum_{j=1}^m p_j X_j = \beta
\end{equation}
for any $\beta$ in the convex hull generated by $X=\{X_1, \cdots, X_m\}$ which is defined by
$$
    \Delta_X = \left\{\sum_{j=1}^m p_j X_j: (p_1, \cdots, p_m)\in \Delta_m\right\}.
$$
When $d=1$, the solution is given by the principle of maximum entropy due to Jaynes \cite{Jaynes}.

Let us define
$$
Z(t)=\sum_{j=1}^m e^{\langle t, X_j \rangle},\qquad t\in \R^d.
$$
We have
$$
  \nabla Z(t) = \sum_{j=1}^m X_j e^{\langle t, X_j \rangle}
$$

\begin{thm} Suppose that the vectors $X_1, \cdots, X_m$ are not coplanar and $\beta$
is in the interior of $\Delta_X$. Then under the constraints
$\sum_{j=1}^m p_j=1$ and $\sum_{j=1}^m p_jX_j=\beta$, the entropy function $h(p)$
attains its maximum at the maximal point $p^*$ defined by
\begin{equation}\label{Gibbs}
   p_j^* = \frac{e^{\langle t_\beta, X_j\rangle}}{Z(t_\beta)} \qquad (j=1, 2, \cdots, m)
\end{equation}
where $t_\beta$ is
the unique solution of the equation
\begin{equation}\label{Eqn-t-beta}
           \frac{\nabla Z(t)}{Z(t)}=\beta.
\end{equation}
Actually the maximal point is unique and the maximum entropy is equal to
$$
      h(p^*) = \log Z(t_\beta) -\langle t_\beta, \beta\rangle.
$$
The map $\beta \mapsto t_\beta$ is $C^\infty$-diffeomorphism from $\stackrel{\circ}{\Delta}_X$
onto $\mathbb{R}^d$.
\end{thm}

The proof of the theorem will be decomposed into several lemmas.

We will denote by $\stackrel{\circ}{A}$ the interior of a set $A$.
Let $M: \mathbb{R}^m \to \mathbb{R}^d$ the linear map defined by the  $d\times m$
matrix
$$
M=\left ( X_1,\dots, X_m \right ).
$$

\begin{lemma}\label{lemma5-1}
$
       \stackrel{\circ}{\Delta}_X = M(\stackrel{\circ}{\Delta}_m).
$
\end{lemma}

\begin{proof} 
Notice that $Mp = \sum_{j=1}^m p_j X_j$.  Let $\pi: \mathbb{R}^m \to \mathbb{R}^m/_{\mbox{\rm ker}(M)}$ be the canonical map from $\mathbb{R}^m$ to the quotient space $\mathbb{R}^m/_{\mbox{\rm ker}(M)}$ and
let $\widetilde{M}$ be the compatible map such that $M
=\widetilde{M}\circ \pi$. Since the map $\pi$ is open and the map $\widetilde{M}$ is a homeomorphism, the subset $M(\stackrel{\circ}{\Delta}_m)$ of $\Delta_X$ is open, so that $M(
\stackrel{\circ}{\Delta}_m)\subset \stackrel{\circ}{\Delta}_X$.

On the other hand, the compact convex set $\Delta_X$ admits its extremal points among
$\{X_1, \cdots, X_m\}$. We first claim that any extremal point, say $X_1$, is a limit point of
$M(\stackrel{\circ}{\Delta}_m)$. In fact, since $X_1-X_j$ ($j=2, \cdots, m$) are in a half-space, there is a vector $h \in \mathbb{R}^d$ such that $\langle h, X_1-X_j\rangle >0$ for all $j=2, \cdots, m$.
Then the points
$$
      Y_n = \frac{\nabla Z(n h)}{Z(n h)} \in M(\stackrel{\circ}{\Delta}_m)
$$
tend to $X_1$ (the argument holds even if some of $X_j$ ($j\ge 2$) are equal to $X_j$, such a case was not excluded). To finish the proof, it suffices to observe that any $\beta \in \stackrel{\circ}{\Delta}_X$ is a convex combination with strict positive coefficients of a set of
points in $\stackrel{\circ}{\Delta}_X$, each of which is sufficiently close to an extremal point
of ${\Delta}_X$.  It follows that $\beta \in  M(\stackrel{\circ}{\Delta}_m)$.
\end{proof}

The constraints (\ref{constraints}) define a compact set on which the entropy function which is continuous attains
its maximum. We will show that the entropy function attains its maximum at an interior point  and the maximal point is unique if $\beta \in \stackrel{\circ}{\Delta}_X$.

\begin{lemma}Assume $\beta \in \stackrel{\circ}{\Delta}_X$. Under the constraints
(\ref{constraints}), the entropy attains its maximum at a point $p^* \in \stackrel{\circ}{\Delta}_m$.
Such maximal points are unique.
\end{lemma}

\begin{proof} The uniqueness of maximal points is just because of the strict concavity of the entropy function.

Let $p^*$ be the maximal point. Suppose that $p^*$ is not strictly  positive.
 Without loss of generality, we assume that $p_j^*>0$ for $j=1,\dots, k$, but $p_j^*=0$ for $j=k+1,\dots, m$.
Since $\beta\in \stackrel{\circ}{\Delta}_X$, by Lemma \ref{lemma5-1}, there exists a  probability vector $q \in \stackrel{\circ}{\Delta}_m$ such that $\beta=Mq$. Denote $r=q-p^*$.
Then we have $Mr=0$.
Notice that $\sum_{i=1}^m r_i=0$ and $r_j>0$ for $j=k+1,\dots, m$.
Consider the perturbation of $p^*$ defined by
$$
p_t={p^*+tr}.
$$
For small $t>0$, $p_t$ is  a probability  satisfying the constraint $Mp_t=\beta$.
Then consider of function $f(t)=h(p_t)$, that is
$$
f(t)=-\sum_{j=1}^m (p^*_j+tr_j)\log(p_j^*+tr_j).
$$
Its derivative is equal to
$$
f'(t)=-\sum_{j=1}^m r_j\left (\log(p_j^*+tr_j)+1\right ).
$$
Let $\epsilon=\min_{1\le j \le k} p_j^*$. Then for $t$ small enough, we have
$$
-r_j(\log(p_j^*+tr_j)+1)\geq |r_j|(\log \epsilon/2+1), \qquad \text{ for } j=1,\dots, k;
$$
but as $t \to 0^+$,
$$
-r_j(\log(p_j^*+tr_j)+1)=-r_j\log(t r_j)-r_j\to \infty, \qquad \text{ for } j=k+1,\dots, m.
$$
So we have  $f'(t)>0$  for small $t>0$ and hence $f(t)$ is increasing near $0$. This contradicts the maximality
of $h(p^*)$.
\end{proof}

\begin{lemma}\label{lem-entropy}
There exists a unique point $t_\beta\in \R^d$ such that
\begin{equation}\label{prob}
p^*_j=\frac{e^{\langle t_\beta, X_j\rangle}}{Z(t_\beta)},
\qquad 1\leq j\leq  m.
\end{equation}
This point $t_\beta$ is the unique solution of the equation $\nabla Z(t)/Z(t) =\beta$.
The maximal entropy $h(p^*)$ is equal to $\log Z(t_\beta) -\langle t_\beta, \beta\rangle$.
\end{lemma}

\begin{proof} Consider the function
$$
     F(p, \lambda, t) = h(p) +\lambda \left(\sum_{j=1}^m p_j -1\right) + \left\langle t, \sum_{j=1}^m p_j X_j-\beta\right\rangle
$$
where $p\in {(\mathbb{R}^{+}) }^m$, $\lambda \in \mathbb{R}$ and $t \in \mathbb{R}^d$. Both
$\lambda$ and $t$ are Lagrange multipliers. The maximal point $p^*$ whose existence is proved above
must be the critical point of $F$. But
$$
    \frac{\partial F}{\partial p_j} = - \log p_j -1 +\lambda + \langle t, X_j\rangle.
$$
We deduce that the maximal point $p^*$ is of the form
\begin{equation}\label{p-t}
       p_j^* = \frac{e^{\langle t, X_j\rangle}}{Z(t)}
\end{equation}
for some $t$ verifying $\nabla Z(t)/Z(t)=\beta$. By the way, we have proved that
the equation $\nabla Z(t)/Z(t)=\beta$ admits a solution.
We claim that there is a unique $t$ verifying (\ref{p-t}). Suppose
$t'\not=t$ is another suitable point. Then
$$
    \frac{e^{\langle t, X_j\rangle}}{Z(t)} = \frac{e^{\langle t', X_j\rangle}}{Z(t')},
    \qquad i.e. \ \ \
        e^{\langle t-t', X_j\rangle} = \frac{Z(t)}{Z(t')}.
$$
Then $e^{\langle t-t', X_i- X_j\rangle}=1$, i.e. $\langle t-t', X_i- X_j\rangle=0$
for all $i, j$. This contradicts that $X_j$'s are not coplanar. Clear
$h(p^*) = \log Z(t_\beta) - \langle t_\beta, \beta\rangle$ where $t_\beta$
is the unique point satisfying (\ref{p-t}).

Now we prove that the equation $\nabla Z(t)/Z(t)=\beta$ admits a unique solution.
Suppose that $t'$ is another solution. We can check that the probability
$p'$ defined by $p'_j = e^{\langle t', X_j\rangle}/Z(t')$ is a maximal point.
So, $p'=p^*$ and then $t'=t$.
\end{proof}

Consider $t_\beta$ as a function of $\beta \in \stackrel{\circ}{C}_X$. It is the inverse function of
$t\mapsto A(t)$ where
$$
A(t)=\frac{\nabla Z(t)}{Z(t)}.
$$

\begin{lemma} The differential $dA(t)$ is non-singular at any point $t\in \mathbb{R}^d$. Hence $\beta \mapsto t_\beta$ is infinitely differentiable.
\end{lemma}
\begin{proof}
Let $p(t)$ be the probability vector defined by $\nabla Z(t)/Z(t)$. Define the $m\times m$
matrix
$$
G(t)=\text{diag~}(p(t))-p(t) \cdot p(t)^T,
$$
where $\text{diag~}(p(t))$ denotes the diagonal matrix with the elements of $p(t)$
 as diagonal elements and $p(t)^T$ denotes the transpose of the column vector $p(t)$.
A direct calculation shows that
$$
dA(t)=MG(t)M^T.
$$

Observe that $G(t)$ is symmetric and it defines the quadratic form
$$
     yG(t)y^T= \sum_{j=1}^m p_j y_j^2 - \left(\sum_{j=1}^m p_j y_j\right)^2.
$$
If we introduce the inner product $(x,y) = \sum_{j=1}^m p_j x_j y_j$,
by the Cauchy inequality we see that $G(t)$ is positive and
$yG(t)y^T=0$ iff $y$ is parallel to $(1, 1, \cdots, 1)$.

Suppose that $dA(t)$ is singular. Then
$xMG(t)M^Tx^T= 0$ for some $x\in \R^d$ with $x\neq 0$, i.e.
$ yG(t)y^T = 0$ for $y=xM$. By the properties of $G(t)$ proved above,
$xM=c(1,1,\dots,1)$ for some $c\not=0$, which means
$\langle x, X_j\rangle =c$ for all $j$, i.e. $X_1,\dots, X_m$ are coplanar.
This is a contradiction. The infinite differentiability is a consequence of the implicit
function theorem.
\end{proof}

Finally, we consider the case that $X_1,\dots, X_m$ are coplanar.
Let $H_0$ be the subspace spanned by $X_i -X_j$ ($i, j= 1, 2, \cdots, m$).
Then $s:=\dim H_0<d$. We can apply the theorem if we replace $\mathbb{R}^d$
by $H_0$ and $t$ by vectors in $H_0$. Then there is a unique $t_\beta$ in $H_0$
associated to $\beta \in \stackrel{\circ}{C}_X$. We can also consider $t \in \mathbb{R}^d$. Then
consider the orthogonal decomposition $\mathbb{R}^d = H_0\oplus H_0^{\bot}$.
For each $\beta \in \stackrel{\circ}{C}_X$, the solution of the equation $
\nabla Z(t)/Z(t) =\beta$ is the set $t_\beta + H_0^{\bot}$ where $t_\beta \in H_0$.

\section{\textbf{Formula of $\gamma(\theta)$ in the coplanar case}}

In this section, we prove an formula for the growth function $\gamma$ when
$X_1, \cdots, X_m$ are coplanar. Let $\eta$ be a non zero vector in $\mathbb{R}^s$,
considered as the normal direction of a hyperplane.
Recall that  $X_1,\dots, X_m$ are $\eta$-coplanar if
\begin{equation}\label{co-plane condition}
\forall j \in \{1,\dots,m\}, \quad
 \langle X_j, \eta \rangle=1.
\end{equation}
Denote by $H_\eta = \{X\in \mathbb{R}^s: \langle X, \eta\rangle =1\}$ the hyperplane
containing $X_1, \cdots, X_m$. By the discussion of the previous section, we have

\begin{lemma}\label{lem-entropy} If $\beta$ is an interior point  in $H_\eta\cap {\mathbf C}_X$,
then the solution $t$ of the equation \eqref{Eqn-t-beta}
exists and
is unique up to a difference of $c\eta$ with $c\in \R$. Moreover, the entropy function
$
h(p)=-\sum_{j=1}^m p_j\log p_j
$
with the constraint
 \begin{equation}\label{constrain}
p_1X_1+\cdots+p_mX_m=\beta.
\end{equation}
attains it maximum at
\begin{equation}\label{prob vector}
p_j=\frac{e^{\langle t, X_j\rangle}}{Z(t)},
\qquad 1\leq j\leq  m.
\end{equation}
\end{lemma}

%


The function $Z(t)$ is conventionally called \emph{partition function} and the probability given
by \eqref{prob vector} is called \emph{Gibbs distribution}.


The vector $t$ depends on $\beta$. In the following, we will consider $\beta=\frac{\theta}{\langle \theta,\eta \rangle}$, so $t$ will depends on $\theta$.

\medskip
\noindent \textbf{Proof of Theorem \ref{thm-entropy}.} For a fixed unit vector $\theta$
in the cone ${\bf C}_{X}$, put $\displaystyle \beta=\frac{\theta}{\langle \theta,\eta \rangle}$,
which is the point on the hyperplane $H_\eta$ of the direction $\theta$.

{\textbf{Lower bound of $\gamma(\theta)$.}}
Take any probability vector $p$ satisfying the constraint \eqref{constrain}.
Let
\begin{equation}\label{n-prime}
n_j=\lfloor np_j \rfloor, \text{ for } j=1,\dots, m-1 \text{ and }n_m=n-(n_1+\cdots+n_{m-1}).
\end{equation}
It is clear that there exists  a constant $c$ independent of $n$ (for example, $c=3m\sum_{j=1}^m\|X_j\|$) such that
\begin{equation}\label{nn-prime}
\left \|\sum_{j=1}^m n_j X_j-n \beta \right \|\le c, \qquad \sum_{j=1}^m |n_j-np_j|\le c.
\end{equation}
So, applying  Theorem \ref{Q(x)} with $C_0=2c$, we have
$$
Q(n\|\beta\|) \m(n\beta)\geq \m\left (\sum_{j=1}^m n_j X_j \right )\ge \frac{n!}{n_1!\dots n_m!}\sim \frac{n!}{(np_1)!\dots (np_m)!},
$$
where $Q(x)$ is the polynomial in Theorem \ref{Q(x)}, and $A\sim B$ means that $A/B$ and $B/A$ are bounded by a polynomial of $n$.
Using Stirling's formula, we obtain that
$$
   \lim_{n\to \infty} \frac{ \log \m(n\beta)}{n}
 \geq   \log \frac{1}{p_1^{p_1}\cdots {p_m}^{p_m}}=
h(p).
$$
Therefore
\begin{eqnarray*}
\gamma(\theta) &=& \lim_{k\to \infty} \frac{ \log \m(k\theta)}{k} =\langle \theta, \eta\rangle
\lim_{n\to \infty} \frac{ \log \m(n\beta)}{n}\geq
\langle \theta , \eta \rangle h(p).
\end{eqnarray*}
 Taking the supremum we get the following lower bound for $\gamma(\theta)$
 $$
\gamma(\theta) \geq  \langle \theta , \eta \rangle  \sup \left \{ h(p): ~ p_1X_1+\cdots+p_mX_m=\beta \right \}.
$$


\indent {\textbf{Upper bound of $\gamma(\theta)$.}}
Let $p$ be a probability vector such that $h(p)$ attains maximum under the restriction \eqref{constrain}.
Let $(n_1,\dots, n_m)$ be the vector defined by \eqref{n-prime}.
Let $x_n=n_1X_1+\cdots+n_mX_m$. Then $|x_n-n\beta|\le c$.
Since $p$ is the Gibbs distribution which is of exponential form (see \eqref{prob vector}), for any $(n_1',\dots, n_m')$ satisfying
$n_1'X_1+\cdots+n_m'X_m=x_n$, we have
$$
p_1^{n_1'}\cdots p_m^{n_m'}=\frac{\exp\langle t, n_1'X_1+\cdots+n_m'X_m\rangle}{Z(t)^n}=\frac{\exp\langle t, x_n\rangle}{Z(t)^n}.
$$

Now we consider a random walk: at any time, we  forward   the step $X_j$ with the probability $p_j$.
Then the above formula says that for any $\omega, \omega'\in \{1, 2, \cdots, m\}^n$, as soon as $\kappa(\omega)=\kappa(\omega')$, both $\omega$ and $\omega'$ have the same probability.
It follows that
$
 Z(t)^{-n} \exp\langle t, x_n\rangle \m(x_n)
 $ is bounded by the  probability that we  arrive at $x_n$ at time $n$, which is bounded by $1$.
Hence,
$$
\m(x_n)\leq \frac{Z(t)^n}{\exp\langle t, x_n\rangle }=\frac{1}{p_1^{n_1}\cdots p_m^{n_m}}.
$$
As $n_j/n\to p_j$, we have
$$
\lim_{n\to \infty} \frac{\log \m(x_n)}{n}\leq h(p).
$$
Finally, since $ \m(n\beta)/\m(x_n)$ is controlled by a polynomial of $n$, we obtain
$$
\gamma(\theta)= \langle \theta , \eta \rangle \lim_{n\to \infty} \frac{\log \m(n\beta)}{n}\leq \langle \theta , \eta \rangle h(p).
$$
This ends the proof of the  theorem. $\Box$

\medskip

\begin{theorem}\label{thm-Z} If $X_1,\dots, X_m$ are $\eta$-coplanar, then for any unit vector $\theta$
in  the interior of ${\bf C}_{X}$ we have
$$
    \gamma(\theta) =\langle \theta , \eta \rangle \log Z(t)-  \langle t, \theta \rangle
$$
where   $t$ is any solution stated in Lemma \ref{lem-entropy}.
\end{theorem}

\begin{proof} As we have seen in the above proof of Theorem \ref{thm-entropy}, the supremum is attained at the
Gibbs distribution $p_j = e^{\langle t, X_j\rangle} Z(t)^{-1}$. Taking the logarithm,
we get
$$
    - \log p_j = \log Z(t) - \langle t, X_j\rangle.
$$
Multiplying both sides by $p_j$ and summing over $j$ allow us to get
$$
    h(p) = \log Z(t) -\sum_{j=1}^m p_j \langle t, X_j\rangle = \log Z(t) - \langle t, \beta\rangle
$$
where $\beta = \theta/\langle \theta, \eta \rangle$.
Finally, we obtain the formula by multiplying $\langle \theta, \eta \rangle$.
\end{proof}

The following proposition is an easy consequence of Theorem \ref{thm-entropy}.

\begin{pro}\label{lem-max} Assume that $X_1,\dots,X_m$ are $\eta$-coplanar. Then
$$\max \frac{\gamma(\theta)}{\langle \eta, \theta \rangle }= \log m  \text{ at }\theta_0=\frac{\sum_{j=1}^m X_j}{|\sum_{j=1}^m X_j|}.
$$
\end{pro}

The reason is that the entropy $h(p)$ attains its maximum $\log m$ at $p_1=\cdots=p_m=1/m$.
The corresponding direction on the hyperplane $H_\eta$ is $\beta_0=\frac{1}{m} \sum_{j=1}^m X_j$, the arithmetic average of $X_1, X_2, \cdots, X_m$. The corresponding unit vector is $\theta_0$.

The growth of the semi-group has been defined as function of the unit vector $\theta$.
If we define the growth $\tilde{\gamma}(\beta)$ as function of the vector $\beta$ located on the
hyperplane $H_\eta$,  we will get a simpler formula
$$
    \tilde{\gamma}(\beta) = \sup \{h(p): p_1 X_1 + \cdots + p_m X_m = \beta\}
    = \log Z(t) - \langle t, \beta\rangle.
$$
This is the conditional variation principle for the multifractal analysis of the Birkhoff
average
\begin{equation}\label{BA}
    \lim_{n\to \infty} \frac{X_{\omega_1}+ X_{\omega_2}+ \cdots + X_{\omega_n}}{n}.
\end{equation}
See \cite{FF, FFW,FLP2008} for discussion in more general case. In general, it is difficult to
determine exactly the possible limits of  the Birkhoff
average. Theorem \ref{thm-entropy} shows that for the special case of (\ref{BA}), the
possible limit is the convex set $H_\eta \cap {\bf C}_{X}$. 

\section{\bf Rigidity (I): Proof of Theorem \ref{thm-XY}}

When the vectors $X_1, \cdots, X_m$ are $\eta$-coplanar, we have proved that
the growth function is equal to
$$
\gamma (\theta) = \langle \theta, \eta\rangle \left ( \log Z(t) - \langle \beta, t\rangle\right )
$$
where $\beta = \frac{\theta}{\langle \theta, \eta\rangle}$ and $t$ is any solution of
the equation $\nabla Z(t)/Z(t)=\beta$. The set of solutions of $t$ is a line
consisting of the points $t_\beta + c \eta$ ($c\in \mathbb{R}$)
and we will choose the one such that the last coordinate of $t_\beta + c \eta$ is zero. It is really
possible that the last coordinate  is zero if $\eta_s\neq 0$. In fact, we can choose
$$
     c = -\frac{\langle t_\beta, e_s\rangle}{\langle \eta, {\mathbf e}_s\rangle}
$$
where ${\mathbf e}_s = (0, ..., 0, 1) \in \mathbb{R}^s$.
The next lemma shows that it is possible to convert the general case to the
case with $\eta_s= : \langle \eta, e_s\rangle \not =0$.

\begin{lemma}\label{lem-tra} Let $X=\{X_1,\dots, X_m\}$ be a set of integral vectors on $\mathbb{Z}^s$ and let $T$ be an invertible $s\times s$ matrix with integral entries.\\
 \indent {\rm (i)} The function $\gamma_{_{TX}}$ is related to $\gamma_{_X}$ by the formula
 $$
 \gamma_{_{TX}}( \theta)=\|T^{-1}\theta\|\gamma_{_X}\left (\frac{T^{-1}\theta}{\|T^{-1}\theta\|}\right ).
 $$
\indent
 {\rm (ii)}  If $X$ is $\eta$-coplanar, then $TX$ is $(T^*)^{-1}\eta$-coplanar, where $T^*$ is the transpose of $T$.
 \end{lemma}

 \begin{proof} (ii) is obvious. (i) is proved by computation. The key point is the observation $\m_{TX}(z)=\m_X(T^{-1}z)$. Then we have
 \begin{eqnarray*}
 \gamma_{_{TX}}( \theta)& = & \lim_{k\to \infty} \frac{\m_{TX}(k\theta)}{k}=
  \lim_{k\to \infty} \frac{\m_{X}(kT^{-1}\theta)}{k} \\
  & =&  \lim_{k\to \infty} \frac{\m_{X}(kT^{-1}\theta)}{\|kT^{-1}\theta\|} \cdot \frac{\|kT^{-1}\theta\|}{k}
  = \|T^{-1}\theta\|\gamma_{_X}\left (\frac{T^{-1}\theta}{\|T^{-1}\theta\|}\right ).
 \end{eqnarray*}
 \end{proof}

Theorem \ref{thm-XY} and Theorem \ref{thm-XY-new} compare two sets of vectors
 $X=\{X_1, \cdots, X_m\}$ and $Y=\{Y_1, \cdots, Y_{m'}\}$, which are respectively
   $\eta$-coplanar and $\eta'$-coplanar. Without loss of generality,  we may assume that
$\eta_s\neq 0$ and $\eta'_s\neq 0$. Otherwise we can consider the images of $X$ and $Y$
under a linear transformation $T$ as stated in the last lemma.
Actually, we may choose $a= (a_1, \cdots, a_{s-1}, 1) \in \mathbb{Z}^s$ such that
$$
    \langle a, \eta\rangle \not= 0, \quad \langle a, \eta'\rangle\not=0.
$$
Such $a$'s do exist, because the condition $\langle z, \eta\rangle = 0$ or $\langle z, \eta'\rangle = 0$
defines a union of two hyperplanes and we can find points $a$ outside these two hyperplanes. Define
$$
     T = \left(
           \begin{array}{ccccc}
             1 &   &        &  & \ -a_1 \\
               &\ 1 \ &        &  & -a_2 \\
               &   & \ \ddots \ &  & \vdots \\
               &   &        & \ 1 \  & -a_{s-1} \\
               &   &        &  & 1 \\
           \end{array}
         \right).
$$
Then we have the last coordinate of $T^{*-1} \eta$ is equal to $\langle a, \eta\rangle \not=0$ and
the last coordinate of
$T^{*-1} \eta'$ is equal to $\langle a, \eta'\rangle \not=0$.

From now on, we assume that $X_1,\dots, X_m$ are $\eta$-coplanar vectors in $\Z^s$ such that
$
\eta_s\neq 0
$
and that there is a (unique) solution $t$ of \eqref{Eqn-t-beta}, i.e.
  \begin{equation}\label{constrain-new}
   \sum_{j=1}^m (X_j-\beta) e^{\langle t, X_j \rangle}=0
\end{equation}
such that $t_s=0$. We define
$$
F(\beta)=\log Z(t)-\langle \beta, t \rangle,
$$
which  is a function on $\beta_1,\dots, \beta_{s-1}$ since $t_s=0$.
The variables $\beta_1, \cdots, \beta_{s-1}$ are independent  and
$(t_1, \cdots, t_{s-1})$ is a $C^\infty$ map of $(\beta_1, \cdots, \beta_{s-1})$.
 Notice that
$F(\beta)=\langle \eta,  \theta \rangle^{-1} \gamma(\theta).$

\begin{lemma}\label{lem-key}
 Let $X_1,\dots,X_m$ be $\eta$-coplanar vectors and $\eta_s\neq 0$. Then
\\
\indent {\rm
(i) }\  $\displaystyle \frac{\partial F}{\partial \beta_j}=-t_j$ for $j=1,\dots, s-1$.
\\
\indent
{\rm (ii)} \  $\displaystyle \sum_{j=1}^{s-1}\frac{\partial \gamma}{\partial \theta_j}\theta_j=\gamma(\theta)-\frac{\eta_s}{\theta_s}\log Z(t).$
\end{lemma}

\begin{proof} (i) Using the chain rule of derivation and  the relation \eqref{constrain-new}, we have
$$
\frac{\partial F}{\partial \beta_j}
  =    Z^{-1}(t)\sum_{\ell=1}^m e^{\langle t, X_\ell \rangle}\sum_{k=1}^{s-1}X_{\ell,k}\frac{\partial t_k}{\partial \beta_j}
 +
\left (-t_j-\sum_{k=1}^{s-1} \beta_k\frac{\partial t_k}{\partial \beta_j}\right )
=- t_j .
$$

(ii) Let us denote $f(\theta)=\langle \eta, \theta \rangle$. Then $\beta=\theta/f(\theta)$ and $\gamma(\theta)=f(\theta)F(\beta)$. By the chain rule of differentiation, we have (${\mathbf e}_j$ denotes the $j$-th element of the canonical basis)
\begin{eqnarray*}
\frac{\partial \gamma}{\partial \theta_j}& = &
\frac{\partial f}{\partial \theta_j} F(\beta)+f\cdot \left (\frac{\partial F}{\partial \beta_1},\dots, \frac{\partial F}{\partial \beta_{s-1}}\right )\times \left (-\frac{\partial f/\partial \theta_j}{f^2}
\left ( \begin{array}{c} \theta_1 \\ \vdots \\ \theta_{s-1}\end{array}\right )
+\frac{1}{f} {\mathbf e}_j
 \right )\\
 &=&  \frac{\partial f}{\partial \theta_j} \left (F(\beta)+\frac{\langle t, \theta \rangle}{f} \right )-t_j\\
 &=& \frac{\partial f}{\partial \theta_j}\log Z(t)-t_j, 
\end{eqnarray*}
where $\times$ stands for the matrix product.
Since $\theta_s^2=1-\sum_{j=1}^{s-1}\theta_j^2$, we have $\displaystyle \frac{\partial f}{\partial \theta_j}=\eta_j-\frac{\eta_s}{\theta_s}\theta_j$,
and hence
\begin{equation}\label{diff}
\displaystyle \frac{\partial \gamma}{\partial \theta_j}=
\left (\eta_j-\frac{\eta_s}{\theta_s}\theta_j\right )\log Z(t)-t_j, \quad j=1,\dots, s-1.
\end{equation}
Therefore
$$\sum_{j=1}^{s-1} \frac{\partial \gamma(\theta)}{\partial \theta_j}\theta_j
=\left (\langle \eta,\theta \rangle -\frac{\eta_s}{\theta_s}\right )\log Z(t)-\langle t,\theta \rangle
=\gamma(\theta)-\frac{\eta_s}{\theta_s}\log Z(t).$$
\end{proof}

\medskip

\noindent \textbf{Proof of Theorem \ref{thm-XY}.}   According to Lemma \ref{lem-tra} and the discussion
just before Lemma \ref{lem-tra},  we may assume without loss of generality that $\eta_s\neq 0$ and $\eta'_s\neq 0$. Otherwise, we consider
$TX$ and $TY$ for some suitable integral   invertible transformation $T$.

First, we claim that $m=m'$. By Proposition \ref{lem-max}, the functions
$\langle \eta,\theta \rangle ^{-1}\gamma_X$ and $\langle \eta,\theta \rangle ^{-1}\gamma_Y$  attain
their respective maximum $\log m$ and $\log m'$.  As $\gamma_X$ and $\gamma_Y$ are the same function, we get
$\log m=\log m'$,  so that $m=m'$.

Next, applying Lemma \ref{lem-key} (i) to $X$, we get
\begin{equation}\label{eq-t1}
 t_j=- \partial   F/\partial \beta_j, \quad j=1,\dots, s-1.
\end{equation}
Similarly, we get
\begin{equation}\label{eq-t2}
\tilde t_j=- \partial \tilde F/\partial \beta_j, \quad j=1,\dots, s-1
\end{equation}
where $$
\tilde F(\beta)= \log \tilde Z(\tilde t)-\langle \beta, \tilde t \rangle,
\qquad \tilde Z(\tilde t)=\sum_{j=1}^m e^{\langle \tilde t, Y_j\rangle}
$$
and  $\tilde t$   is the solution of
$\nabla \tilde Z(t)/\tilde Z(t) = \beta$ such that $\tilde t_s=0$.

Since   $\gamma_X(\theta)=\gamma_Y(\theta)$,
$F(\beta)=\tilde F(\beta)$
so that
 $t=\tilde t$  
 by \eqref{eq-t1} and \eqref{eq-t2}.
 Therefore
$Z(t)=\tilde Z(t)$, \textit{i.e.},
$$
\sum_{j=1}^m e^{\langle t, X_j\rangle}=\sum_{j=1}^m e^{\langle t, Y_j\rangle}.
$$
In this equality, $t$ is a function of  $\widehat{\beta} = (\beta_1, \cdots, \beta_{s-1})$
and $\widehat{\beta}$ varies in a open set of $\mathbb{R}^{s-1}$.
The function $\widehat{\beta} \mapsto t$ is actually a diffeomorphism. So, $t$ varies in a open set $U$ of $\mathbb{R}^{s-1}$.


Consider the polynomial
$$
P(z_1, \cdots, z_{s-1}):=
\sum_{j=1}^m z^{\widehat{X}_j} - \sum_{j=1}^m z^{\widehat{Y}_j}= \sum_{j=1}^m \prod_{k=1}^{s-1} z_k^{X_{j,k}}- \sum_{j=1}^m \prod_{k=1}^{s-1} z_k^{Y_{j,k}}
$$
where $z=(z_1, \cdots, z_{s-1})$ and $\widehat{X}_j = (X_{j, 1}, \cdots, X_{j, s-1})$.
The above proved equality $Z(t) =\tilde Z(t)$ means that $P(z_1, \cdots, z_{s-1})=0$
when $z_k=e^{t_k}$  for $k=1,\dots, s-1$. The polynomials
$\sum_{j=1}^m z^{\widehat{X}_j}$ and $\sum_{j=1}^m z^{\widehat{Y}_j}$ being equal in the open set
 $\{(e^{t_1},\dots, e^{t_{s-1}}): (t_1, \cdots, t_{s-1})\in U\}$, the exponents $\{\widehat{X}_1,\dots, \widehat{X}_m\}$ must be a permutation of the exponents $\{\widehat{Y}_1,\dots, \widehat{Y}_m\}$.
 Therefore $\{X_1, \cdots, X_m\}$ is a permutation of
 $\{Y_1, \cdots, Y_m\}$.
$\Box$

\begin{remark}
Here is an another argument for proving the equality  $Z(t) =\tilde Z(t)$. Notice that $F$
and $\tilde F$ are Legendre transforms of the convex functions $\log N$ and $\log \tilde N$.
Then  $\log N$ and $\log \tilde N$ are the Legendre transforms of $F$ and $\tilde F$.
Since $F=\tilde F$, we get $\log N = \log \tilde N$ so that $N = \tilde N$.

\end{remark}

\medskip

\section{\bf Rigidity (II): Proof of Theorem \ref{thm-XY-new}}

It is assumed that ${\mathbf C}_X={\mathbf C}_Y$.
We shall denote the common cone by ${\mathbf C}$.

If we do not have the information that $\eta'$ is a multiple of $\eta$,  it may happen that
$t\neq \tilde t$ where $t$ and $\tilde t$ correspond to the same $\beta$ (see the last section for notation).
We will use another correspondence between $\beta$ and the solutions $t$
of $\nabla Z(t)/Z(t) = \beta$.


\subsection{Standard solution}
We call a solution $t$ in Lemma \ref{lem-entropy} a \emph{standard solution} if $Z(t)=1$.

\begin{lemma} 
 If $t$ is a solution in Lemma \ref{lem-entropy}, then  the standard solution is
$$t'=t-\left (\log Z(t)\right )\eta.$$

\end{lemma}

\begin{proof} 
Let $t$ be a solution in Lemma \ref{lem-entropy}. Let $t'=t+d\eta$.
Then
$Z(t')=e^dZ(t)$
and $ Z(t')=1$ when $d=-\log Z(t)$.
 \end{proof}

\begin{lemma}\label{thm-standard} Let $X$ and $Y$ are two collections of coplanar vectors with $\eta_s\neq 0$ and $\eta'_s\neq 0$.
If $\gamma_{_X}=\gamma_{_Y}$, then
for any $\theta$ belongs to the interior of ${\mathbf C}$, they have the same standard solution, \textit{i.e.},
$$t'=(\tilde t)'.$$
\end{lemma}

\begin{proof}  Let $t$ be the solution in Lemma \ref{lem-entropy} with $t_s=0$, then the $s$-th coordinate of the corresponding standard solution is $t'_s=-\eta_s \log Z(t)$. Similarly, we have $(\tilde t)'_s=- \eta'_s \log \tilde Z(\tilde t)$.
Hence, by Lemma \ref{lem-key} (ii), we get
$$
t'_s=-\eta_s \log Z(t)=-\eta'_s \log \tilde Z(\tilde t)=({\tilde t})'_s.
$$
 For $j=1,\dots, s-1$, using \eqref{diff},
$$
 t'_j=
t_j-\eta_j\log Z(t)=-\frac{\partial \gamma}{\partial \theta_j}-\frac{\theta_j}{\theta_s}\eta_s\log Z(t)
= -\frac{\partial \gamma}{\partial \theta_j}+ \frac{\theta_j}{\theta_s}t'_s.
$$
 So $t'_j=(\tilde t)'_j$.
\end{proof}

Let $z=(z_1,\dots, z_s)$ and $n=(n_1,\dots,n_s)$.  Denote $z^n=z_1^{n_1}\cdots z_s^{n_s}$, and define
$$
P(z)=\sum_{j=1}^m z^{X_j}, \quad Q(z)=\sum_{j=1}^{m'} z^{Y_j}.
$$

\begin{lemma}\label{manifold} Under the assumptions of Lemma \ref{thm-standard},
the algebraic equations
$P(z)=1$ and $Q(z)=1$ have infinitely many common solutions.
\end{lemma}

\begin{proof} Notice that $P(e^t)= Z(t)$ and $Q(e^t) = \tilde Z(t)$ where
$e^{t}:=(e^{t_1},\dots, e^{t_s})$.
For any $\theta$ belongs to the interior of ${\mathbf C}$, let $t'$ be the
corresponding standard solution. Then   $e^{t'}=e^{(\tilde t)'}$ is a common solution
of $P(z)=1$ and $Q(z)=1$.
\end{proof}

\subsection{$H_\eta$ and $H_{\eta'}$ are parallel when $s=2$.}
We will reduce Theorem \ref{thm-XY-new} to Theorem \ref{thm-XY}. The key point is to show that $H_\eta$ and $H_{\eta'}$ are parallel.
Essentially  we only need to prove the parallelism in the two-dimensional case and the general
case will be reduced to this special case. In the following, we only need the above lemmas
in the two-dimensional case.

First, we recall some basic definitions and facts on algebraic plane curve.

An \emph{algebraic plane curve} is a curve consisting of the points of the plane whose
coordinates $x,y$ satisfy an equation $f(x,y)=0$ for some $f\in \R[x,y]$ . The curve is said to be {\em irreducible} if  $f$ is an irreducible polynomial.
It is well-known the polynomial ring $\R[x,y]$ is a unique factorization domain, that is,
any polynomial $f$ has a unique factorization
$f=f_1\dots f_n$ (up to constant multiples) as a product of irreducible factors $f_j$.
Hence, every algebraic curve is a union of several irreducible algebraic curves.
The following lemma is fundamental, see for example \cite{Sha}.

\begin{lemma}[\cite{Sha}]\label{lem-curve} Let $f\in \R[x,y]$ and $g\in \R[x,y]$ be two polynomials. Suppose that $f$ is irreducible polynomial and is not a factor of $g$. Then the system of equations $$
f(x,y)=g(x,y)=0$$ has only a finite
number of solutions.
\end{lemma}

Let $f(x,y)$ be a polynomial. A \emph{highest term} of $f$ is a
 term of $f$ whose degree is  equal to  the degree of $f$.
The homogenous polynomial consisting of all the highest terms of $f$
will be denoted by $H_f$. It is  called the \emph{principal part}
of $f$.
For example, if $f(x,y)=ax^2+bxy+cy^2+dx+ey+f$, then $H_f(x,y)=ax^2+bxy+cy^2$.

\begin{theorem}\label{thm-2-dim}  Suppose  $X=\{X_1,\dots, X_m\}\subset \Z^2$ is $\eta$-coplanar and
$Y=\{Y_1,\dots, Y_{m'}\}\subset \Z^2$ is $\eta'$-coplanar, and that    $X$ and $Y$
  define the same directional growth function. Then
 $\eta=c\eta'$ for some $c>0$. In other words, the line containing $X$ and the line containing $Y$ are parallel.
\end{theorem}

\begin{proof} By applying a suitable linear transformation, we may assume that
${\mathbf C}=(\R^+)^2$ and
  $X_1,\dots, X_m, Y_1,\dots, Y_{m'}\in \N^2$ (see Lemma \ref{lem-tra}.)
We can further assume that
$\eta=(1/n,1/n)$ for some integer $n\geq 1$. Indeed, let $(x_1,y_1)$ and $(x_2,y_2)$ be two vectors
in $X$ locating on two boundary rays of ${\mathbf C}={\mathbf C}_X={\mathbf C}_Y$, respectively.
Let $T$ be the linear transformation which maps $(x_1,y_1)$ and $(x_2,y_2)$ to $(1,0)$ and $(0,1)$, respectively.
Clearly $T$ is invertible, and the entries of $T$ belong to $\mathbb Q$. It follows that $T(x,y)\in {\mathbb Q}^2$ for
all $(x,y)\in X\cup Y$. Hence, there exists a positive integer $n$ such that $nT$ is an integral matrix, and
$nT(x,y)$ are integral vectors for all $(x,y)\in X\cup Y$. It is seen that $nT$ is the desired transformation.

Our aim is then  to show that $\eta'=(c,c)$ for some $c>0$.
Consider the polynomials
$$
P(x,y)=\sum_{j=1}^m x^{X_{j,1}}y^{X_{j,2}},\quad Q(x,y)=\sum_{j=1}^{m'} x^{Y_{j,1}}y^{Y_{j,2}}.
$$
By Lemma \ref{manifold}, $P(x, y)-1=0$ and $Q(x, y)-1=0$ have infinitely common roots. Hence,
by Lemma \ref{lem-curve},
$P(x,y)-1$ and $Q(x,y)-1$ have a common non-trivial factor. Let us denote their greatest common
factor by $S(x,y)$.

We claim that the principal part $H_S(x,y)$ of $S(x,y)$ contains at least two terms.
Let $P(x,y)-1=S(x,y)T(x,y)$. Since $P(x,y)$ is homogenous,
$$H_S(x,y)\cdot H_T(x,y)=H_P(x,y)=P(x,y).$$
Suppose that $H_S(x,y)$ is a monomial, say, $H_S(x,y)=x^py^q$.   Then it must divides each term of $P(x,y)$,
including $x^n$ and $y^n$ (Notice that $x^n$ and $y^n$ are two terms in $P(x,y)$ with the coefficients different from $0$).
It is impossible and our claim is thus proved.

Let $Q(x,y)-1=S(x,y)R(x,y)$. Then
$$
H_S(x,y)\cdot H_R(x,y)=H_Q(x,y).
$$
So $H_Q(x,y)$ has at least two terms. Therefore, $Q(x,y)$ has two terms with the same degree. That is to say,
there exist two different integers $i$ and $j$ such that
$$
(Y_{i,1}, Y_{i, 2})\not= (Y_{j,1}, Y_{j, 2}), \qquad
   Y_{i,1}+ Y_{i, 2} = Y_{j,1}+ Y_{j, 2}.
$$
Recall that
$$
    \eta'_1 Y_{i,1} + \eta'_2 Y_{i, 2} =1, \quad \eta'_1 Y_{j,1} + \eta'_2 Y_{j, 2} =1.
$$
Then solving $\eta_1'$ and $\eta'_2$ as unknown leads to
$$
   \eta'_1 = \frac{Y_{j,2} - Y_{i, 2}}{D}
    = \frac{Y_{i,1} - Y_{j, 1}}{D} = \eta'_2
$$
where $D = Y_{i,1} Y_{j,2} - Y_{i,2} Y_{j,1}\not=0$.
Thus we have proved $\eta'=c\eta$ for some $c$.
\end{proof}

\subsection{Proof of Theorem \ref{thm-XY-new}}
The proof of Theorem \ref{thm-XY-new} is decomposed into several lemmas.


 \begin{lemma}\label{lem-ite} Let $X=\{X_1,\dots, X_m\}$ and    $X^{(p)}$ denotes the $p$-th iteration of $X$. Then

 {\rm (i)}  $X$ and $X^{(p)}$ defines the same directional growth function.

 {\rm (ii)} If $X$ is $\eta$-coplanar, then $X^{(p)}$ is $\eta/p$-coplanar.
 \end{lemma}

 \begin{proof} (ii) is obvious. In the following, we prove (i).
 Let us denote
 $$
 X^{(p)}=\{W_1,\dots, W_{m^p}\}
 $$
 which is the set of all $\sum_{j=1}^p X_{\omega_j}$ with $\omega_1\dots\omega_p\in \{1,\dots,m\}^p$.
 Notice that many vectors in $X^{(p)}$ are repeated. For example, $X_1+ X_2 + \cdots+ X_m$ and
 $X_m +\cdots + X_2 + X_1$ are the same.
 Let
 $${\CJ}=X_1\N+\cdots+X_m\N, \qquad
 {\mathcal J}'=W_1\N+\cdots+W_{m^p}\N.
 $$
 Then ${\mathcal J}'\subset \CJ$ and $\CJ'$ is relatively dense in $\CJ$. Moreover, take any $z\in \CJ'$ and let
 $\omega_1\dots \omega_n\in \{1,\dots,m\}^n$ be a path relative to the walk guided by $X$ such that
 $$
 z=\sum_{j=1}^n  X_{\omega_j}.
 $$
 Then $n$ must be a multiple of $p$. Write
 $$
 z=\sum_{k=1}^{n/p} \sum_{j=1}^p X_{\omega_{j+(k-1)p}}.
 $$
 Hence $(\omega_1\dots \omega_p)(\omega_{p+1}\dots \omega_{2p})\dots (\omega_{n-p+1}\dots\omega_n)$ is  a path  relative to the walk guided by $X^{(p)}$. Therefore
 $
 {\mathbf m}_X(z)={\mathbf m}_{X^{(p)}}(z)
 $ for all $z\in {\mathcal J}'$.
 It follows that they define the same direction growth function.
 \end{proof}

Let us  recall some notions on convex set (see \cite{Rock}).
A \emph{face} of a convex set $C$ is a convex subset $C'$ of $C$ such that every closed line segment in $C$ with a relative
interior point in $C'$ has both end points in $C'$. A vertex (i.e. an extremal point) of a convex set $C$ is regarded as a $0$-dimensional face. A face of dimension $1$ is conventionally called an \emph{edge}.

If $C={\mathbf C}_X$, then the only $0$-dimensional face is the origin, a $1$-dimensional face
is also called an \emph{extreme  ray}.

A set is called a \emph{polytope} if it is the convex hull of finitely many points.

\medskip

\begin{lemma} Under the condition of Theorem \ref{thm-XY-new},
we have $\eta=c\eta'$ for some $c>0$.
\end{lemma}

\begin{proof}

Let $D$ be a two-dimensional face of   ${\mathbf C}$.
Let $X_D=X\cap D$ be the collection of $X_j$ which locates in $D$. Similarly, we define $Y_D$.
Clearly $X_D$ and $Y_D$ are coplanar collections.

 Let  us define a Frobenius problem with defining data
$X_D$. Let ${\mathcal J}'$ be the semi-group generated by $X_D$.
For $z\in {\mathcal J}'$, as before, we define ${\mathbf m}_{X_D}(z)$ to be the number of paths terminated at $z$.
 Since $D$ is a face, $X_{i_1}+\cdots+X_{i_k}$ is in $D$ only if
 all the vectors $X_{i_j}$ belongs to $X_D$, so ${\mathbf m}_{X_D}(z)={\mathbf m}_{X}(z)$ for all $z\in {\mathcal J}'$.
For $x\in {\mathbf C}_{X_D}$, we define
${\mathbf m}_{X_D}(x)$ to be the multiplicity of the element in $\CJ'$ nearest to $x$.
By Theorem \ref{Q(x)}, the ratio of ${\mathbf m}_X(x)$ and
${\mathbf m}_{X_D}(x)$ is controlled by a polynomial on $\|x\|$.
Now, we define the directional growth function
$$
\gamma_{X_D}(\theta)=\lim_{k\to \infty} \frac{\log {\mathbf m}_{X_D}(k\theta)}{k}.
$$
Then $\gamma_{_{X_D}}(\theta)=\gamma_{_X}(\theta).$
Similarly, we define a Frobenius problem with defining data $Y_D$. It turns out that
 $$
 \gamma_{_{X_D}}(\theta)=\gamma_{_X}(\theta)=\gamma_{_Y}(\theta)=\gamma_{_{Y_D}}(\theta).
 $$

%
Therefore, by Theorem \ref{thm-2-dim}, the line containing $X_D$ is parallel to the line containing $Y_D$,
so there is a constant $c_D$ such that $y=c_Dx$ as soon as
$x\in X_D$ and $y\in Y_D$ are located on a same line.

Clearly $K_\eta={\bf C}_X\cap H_\eta$ is the polytope generated by $X$. Let $V_\eta$ be the vertex set of $K_\eta$.
Clearly $V_\eta\subset X$. Similarly we define $K_{\eta'}$ and $V_{\eta'}$.

For any $x\in V_\eta$, there is a point $y\in V_{\eta'}$ such that $x$ and $y$ are located on
a same extremal ray of ${\mathbf C}$. Let $c_x>0$ be the real number such that $y=c_xx$.

 If $x$ and $x'$ are the two end points of an edge of $K_\eta$,
 this edge determines a two dimensional face $D$ of ${\mathbf C}$. Then $c_x=c_{x'}=c_D$ by the above discussion.
Take  two arbitrary points $x,x''\in V_\eta$, there always exists a path (consisting of edges of $K_\eta$)
from $x$ to $x''$. We deduce that  $c_x=c_{x''}$.

Let us denote this common constant by $c$. Then
for any $y\in V_{\eta'}$,
$$
\langle y, c^{-1}{\eta} \rangle =\langle cx, c^{-1}{\eta} \rangle = \langle x, {\eta} \rangle =1.
$$
Since the points in $V_{\eta'}$ span the space $\R^d$,
we have $c^{-1}{\eta} = \eta'$, i.e. $\eta=c\eta'$.
\end{proof}

\begin{lemma} The constant $c$ in the last lemma is a rational number.
\end{lemma}

\begin{proof}
Since the points in $X$ are integral, solving $\langle X_j,\eta\rangle=1$ by Cramer rule, we get the solution   $\eta$, whose entries  are rational numbers. The entries of $\eta'$ are also rational numbers.
It follows that
 $c$ is a rational number.
 \end{proof}

{\it Proof of Theorem \ref{thm-XY-new}.}
Let us write $c=q/p$.
Consider $X^{(q)}$, the $q$-th iteration of $X$,  and $Y^{(p)}$, the $p$-th iteration of $Y$.
Both of them are $(\eta/q)$-coplanar  and  define the same directional growth function by Lemma \ref{lem-ite}.
Therefore,  by Theorem \ref{thm-XY},  $X^{(q)}$ is a permutation of  $Y^{(p)}$.
$\Box$

\bigskip

\noindent \textbf{Acknowledgement.} The authors would like to thank Jia-yan Yao and Alain Rivi\`ere
for helpful discussions.


\begin{thebibliography}{99}
\bibliographystyle{ieee}
\addcontentsline{toc}{chapter}{Bibliography}

\bibitem{Alf05} J. Ramirez Alfonsin, {\it The Diophantine Frobenius problem}, Oxford Univ. Press,~2005.

\bibitem{Car} R. D. Carmichael, {\it On sequences of integers defined by recurrence relations,} Quart. J. Math. \textbf{41} (1920), 343--372.

\bibitem{CP} D.~Cooper and T.~Pignataro, {\it On the shape of Cantor
sets}, J. Differential Geom., \textbf{28} (1988), 203--221.

\bibitem{DS}
G.~David and S.~Semmes, {\it Fractured fractals and broken dreams :
self-similar geometry through metric and measure}, Oxford Univ.\
Press,~1997.

\bibitem{FaMa92} K. J.~Falconer and D. T.~Marsh, {\it On the Lipschitz equivalence of
Cantor sets}, Mathematika, \textbf{39} (1992), 223--233.

\bibitem{FF} A. H.~Fan and D. J. ~Feng,  {\it On the distribution of
  long-term time averages on symbolic space}, { J.~Statist.\
    Phys.,} \textbf{ 99} (2000), ~813--856.

\bibitem{FFW}
\newblock A. H. Fan, D. J. Feng and J. Wu,
\newblock \emph{Recurrence, dimension and entropy},
\newblock J. London Math. Soc., \textbf{64} (2001), 229--244.

\bibitem{FLP2008} { A. H.~Fan, L. M.~Liao and J.~Peyri\`ere},
  {\it Generic points in systems of specification and Banach valued
  Birkhoff ergodic average}, {Discrete \& cont.
    dyn. sys.,} \textbf{21} (2008), ~1103--1128.


\bibitem{Jaynes}
E. T. Jaynes,  {\it Information Theory and Statistical Mechanics},
 Physical Review. Series II 106 (4) (1957), 620-630.



\bibitem{Lau13} J. J. Luo and K. S. Lau, {\it Lipschitz equivalence of self-similar sets and hyperbolic boundaries},
Adv. Math. \textbf{235} (2013), 555--579.


\bibitem{RRW12}
H. Rao, H. J. Ruan, and Y. Wang,  {\it Lipschitz equivalence of Cantor sets and algebraic properties of contraction ratios},  Trans. Amer. Math. Soc.,  \textbf{364} (2012), 1109-1126.

\bibitem{RRW13} H. Rao, H. J. Ruan and Y. Wang,
{\it Lipschitz equivalence of self-similar sets: algebraic and geometric properites},
Contemp. Math., \textbf{600} (2013).

\bibitem{RZ14} H. Rao and Y. Zhang,  {\it  Higer dimensional Frobenius problem and Lipschitz equivalence of Cantor sets}, Preprint 2014.

\bibitem{Rock}
R. T. Rockafellar, {\it Convex Analysis}, Princeton University Press, Princeton, 1970.


\bibitem{Sha} I. Shafarevich, {\it Basic Algebraic Geometry}, Second edition, Springer-Verlag, Berlin, 1994.

\bibitem{Syl} J. J. Sylvester, \textit{Mathematical questions with their solutions,} Education Times \textbf{41-21} (1884).

\bibitem{XiXi13} L.F. Xi and Y. Xiong,  \emph{Lipschitz Equivalence Class, Ideal Class and the Gauss Class Number Problem.}
Preprint 2013 (arXiv:1304.0103 [math.MG]).

\end{thebibliography}
\end{document}